\documentclass[11pt,a4paper,usanames,dvipsnames,reqno]{amsart}

\usepackage[margin=0.8in]{geometry}

\usepackage{scrtime} 
\usepackage{lmodern} 

\usepackage{amsthm} 
\usepackage{amsmath}
\usepackage{amsfonts}
\usepackage{amssymb}
\usepackage{tikz-cd}
\usepackage{graphicx}
\usepackage{caption} 
\usepackage{subcaption}
\usepackage{float}
\usepackage{cleveref} 
\usetikzlibrary{arrows}
\usepackage{mathtools} 
\usepackage{dsfont}
\usepackage{color}
\usepackage{xfrac}
\usepackage{enumitem}

\newcommand{\G}{\mathbf G} 
\newcommand{\Gt}{\widetilde{\G}}  
\newcommand{\subG}{\mathbf H} 
\newcommand{\W}{\mathbf W} 
\newcommand{\Wt}{\widetilde{\W}}  
\newcommand{\T}{\mathbf T} 
\newcommand{\m}{\widetilde{m}}  
\newcommand{\EG}{\mathbf G^-} 
\newcommand{\EM}{\mathbf W^-}
\newcommand{\Em}{m^-} 
\newcommand{\EA}{\alpha^-}
\newcommand\VG{\mathbf G^+} 
\newcommand{\VM}{\mathbf W^+}
\newcommand{\Vm}{m^+} 
\newcommand{\VA}{\alpha^+}

\newcommand{\D}{D} \newcommand{\SG}{\mathcal{S}}
\newcommand{\MSG}{\mathcal{MS}} 
 
\DeclareMathOperator{\Tr}{Tr} \DeclareMathOperator{\ind}{ind} 
\newcommand{\eqLapl}{\prescript {\chi \!}{}{\Delta}^{\!\Wt}{}}

\newcommand{\R}{\mathbb{R}} 
\newcommand{\C}{\mathbb{C}} 
\newcommand{\Z}{\mathbb{Z}} 
\newcommand{\Torus}{\mathbb{T}}
\newcommand{\UVP}{\bigcup_{\alpha\in\mathcal{A}(\G)}}

\newtheorem{theorem}{Theorem}[section]
\newtheorem*{theorem*}{Theorem}
\newtheorem{proposition}[theorem]{Proposition}

\newtheorem{corollary}[theorem]{Corollary}

\theoremstyle{definition} 
\newtheorem{definition}[theorem]{Definition}
\newtheorem{example}[theorem]{Example}
\newtheorem{remark}[theorem]{Remark}

\theoremstyle{plain}

{\bf}{\it}
\theoremstyle{definition}

\numberwithin{equation}{section}
\usepackage{accents} 
\usepackage{bbm}
\newcommand{\myparagraph}[1]{\noindent\textbf{#1}}

\newcommand{\Betti}{b} 

\newcommand{\1}{\mathbbm 1} 

\newcommand{\bd} {\partial} 
\DeclareMathOperator*{\dcup} {\mathaccent\cdot\cup} 

\newcommand{\e}{\mathrm e} 
\newcommand{\im}{\mathrm i} 

\newcommand{\dd} {\, \mathrm d} 

\newcommand{\orient}[1]{\accentset{\curvearrowright}{#1}}

\newcommand{\normsymb}{\|} 

\newcommand{\quadtext}[1]{\quad\text{#1}\quad}
\newcommand{\qquadtext}[1]{\qquad\text{#1}\qquad}

\newcommand{\normsqr}[2][{}]{\normsymb{#2}\normsymb^2_{#1}} 

\newcommand{\iprod}[3][{}]{\langle{#2},{#3}\rangle_{#1}} 

  \newcounter{lookcounter} 

\title{Spectral gaps and discrete magnetic Laplacians}%

\author{John Stewart Fabila-Carrasco} %
\address{Department of Mathematics, University Carlos III de Madrid,
  Avda. de la Universidad 30, 28911. Leganés (Madrid), Spain.}
\email{jfabila@math.uc3m.es}

\author{Fernando Lled\'o} %
\address{Department of Mathematics, University Carlos III de Madrid,
  Avda. de la Universidad 30, 28911. Leganés (Madrid), Spain and Instituto de Ciencias Matemáticas (CSIC-UAM-UC3M-UCM), Madrid.}
\email{flledo@math.uc3m.es}

\author{Olaf Post} %
\address{Fachbereich 4 -- Mathematik, Universit\"at Trier, 54286
  Trier, Germany} \email{olaf.post@uni-trier.de}

\thanks{JSFC was supported by Spanish Ministry of Economy and
Competitiveness through project DGI MTM2014-54692-P}

\thanks{FLl was supported by Spanish Ministry of Economy and
Competitiveness through project DGI MTM2014-54692-P and the
\emph{Severo Ochoa} Program for Centers of Excellence in R\&D
(SEV-2015-0554).}

\thanks{OP acknowleges support from the EPSRC funded research network
  ``Analysis on Graphs''.}

\begin{document}
\begin{abstract}
  The aim of this article is to give a simple geometric condition that
  guarantees the existence of spectral gaps of the discrete Laplacian
  on periodic graphs.  For proving this, we analyse the discrete
  magnetic Laplacian (DML) on the finite quotient and interpret the
  vector potential as a Floquet parameter.  We develop a procedure of
  virtualising edges and vertices that produces matrices whose
  eigenvalues (written in ascending order and counting multiplicities)
  specify the bracketing intervals where the spectrum of the Laplacian
  is localised. We prove Higuchi-Shirai's conjecture for $\Z$-periodic
  trees and apply our technique in several examples like the
  polypropylene or the polyacetylene to show the existence of spectral
  gaps.
\end{abstract}

\subjclass[2010]{05C50, 47B39, 47A10, 05C65}

\keywords{Discrete magnetic Laplacian, spectral gaps on periodic
graphs, Laplacian on graphs, spectral ordering}

\maketitle

%
%
\section{Introduction}
\label{Introduction}
%

The analysis of Schr\"odinger operators (in particular Laplacians) and
their spectra on periodic structures is one of the most important
features in solid state physics.  Periodicity here means that there is
a discrete group $\Gamma$ (typically Abelian) acting on the underlying
structure, e.g., a manifold or a graph, with compact quotient that
commutes with the operator. Using Floquet theory, the analysis of the
operator can be reduced to the analysis of a related family of
operators on the quotient.  In the case of graphs, the family of
operators corresponds to a family of finite dimensional operators,
i.e., matrices.

Let $\Gt$ be a $\Gamma$-periodic discrete graph with quotient graph
$\G$.  The aim of the present article is to present simple geometric
conditions on $\G$ that imply that the Laplacian $\Delta^{\Gt}$ (with
standard, combinatorial or general periodic weights) on the infinite
periodic graph $\Gt$ has not the maximal possible interval as
spectrum.  Such Laplacians are said to have a \emph{spectral gap}. To
show the existence of spectral gaps we develop a purely discrete
bracketing technique based on \emph{virtualisation} of edges and
vertices on $\G$ for the discrete magnetic Laplacian on $\G$.  In the
context of periodic manifolds or metric graphs Dirichlet-Neumann
bracketing allows to localise the spectrum of the differential
operator in certain closed intervals whose end points are specified by
the Laplacian on a fundamental domain with Dirichlet or Neumann
boundary conditions (see,
e.g.,~\cite{lledo-post:08b,lledo-post:08,lledo-post:07} and references
cited therein).  Our method of virtualisation of edges and vertices
can be seen as a discrete version of Dirichlet-Neumann bracketing.

An important ingredient of our analysis is the discrete magnetic
Laplacian (DML). It is a natural generalisation of the usual Laplacian
and incorporates the presence of a magnetic field on the graph.
Typically the magnetic field enters in the analysis via a vector
potential, which is a function of the edges $\alpha\colon E\to
\R/2\pi\Z$. Discrete magnetic Laplacians on covering graphs with
general discrete group actions have already been treated,
e.g.,~in~\cite{sunada:94c} and~\cite{mathai-yates:02,msy:03}, also for
general periodic magnetic fields.  Magnetic Laplacians and Laplacians
on Abelian covering graphs are discussed also
in~\cite{higuchi-shirai:99,higuchi-shirai:04}.  Korotyaev and Saburova
treated discrete Schr\"odinger operators on discrete graphs in a
series of articles (see,
e.g.,~\cite{koro-sabu:14,koro-sabu:15,koro-sabu:17} and references
therein).

In~\cite{koro-sabu:15} the authors develop a bracketing technique
similar to the Dirichlet-Neumann bracketing mentioned before and
proved estimates of the position of the spectral bands for the
combinatorial Laplacian in terms of suitable Neumann and Dirichlet
eigenvalue intervals.  Moreover, they give an upper estimate of the
total band length in terms of these eigenvalues and some geometric
data of the graph; this method is extended in~\cite{koro-sabu:17} to
the case of magnetic Laplacians with periodic magnetic vector
potentials.  A method to open gaps in the spectrum --- completely
independent of the periodicity of the underlying graph --- was
presented by Schenker and Aizenman in~\cite{aizenman-schenker:00},
decorating each vertex of the original graph with a copy of a given
finite graph.

Periodic graphs can also be understood as covering graphs (see,
e.g.,~\cite{sunada:08,sunada:13}).  The existence of spectral gaps is
related with the \emph{full spectrum conjecture} stating that the
discrete Laplacian on a maximal Abelian covering $\pi \colon \Gt \to
\G=\Gt/\Gamma$ has maximal possible spectrum.  A covering is
\emph{maximal Abelian} if the covering group is $\Gamma=\Z^b$, where
$b=\Betti(\G)=|V(\G)|-|E(\G)|+1$ is the first Betti number of $\G$
(see~\cite[Chapter~6]{sunada:13} for details; as usual $V(\G)$ and
$E(\G)$ denote the set of vertices and edges of $\G$, respectively).
The conjecture states that the (standard) Laplacian on a maximal
Abelian covering has maximal possible spectrum $[0,2]$, i.e., no
spectral gaps (see~\cite[Conjecture~3.5]{higuchi-shirai:99}) --- at
least when there is no vertex of degree $1$. This conjecture was
partially solved by~\cite[Proposition~3.6 and 3.8]{higuchi-shirai:99}
for graphs where all vertices have \emph{even} degrees, and certain
regular graphs with odd degrees. Graphs with even vertex
degrees allow so-called \emph{Euler paths}, related to the famous
K\"onigsberg bridges problem due to Euler.  In addition, Higuchi and
Shirai~\cite{higuchi-shirai:04} also realised that one needs
additional conditions for the conjecture to be true; namely that there
is no vertex of degree $1$.  We confirm this conjecture for periodic
trees (i.e., when the quotient graph has Betti number $1$).  Moreover,
Higuchi and Nomura~\cite{higuchi-nomura:09} show that the normalised
Laplacian on a maximal Abelian covering graph of a finite even-regular
graph respectively odd-regular and bipartite graph has absolutely
continuous spectrum and no eigenvalues.

The article is structured as follows: in Section~\ref{sec:DML} we
recall basic definitions and results on discrete weighted graphs and
DMLs. The use of weighted graphs is particularly well adapted to the
virtualisation procedure described below, since, for example, the
virtualisation of an edge $e$ can also be understood by changing the
edge weight to $m_e=0$. In Section~\ref{sec:spectral_order} we develop
a discrete bracketing technique for finite weighted graphs which is
based on the manipulation of the (finite) fundamental domain of the
periodic discrete graph.  This technique is based on a selective
virtualisation of certain edges $E_0\subset E$ and vertices
$V_0\subset V$ of a given weighted graph $\W=(\G,m)$ with vector
potential $\alpha$ and weights $m$ on $V$ and $E$, respectively. The
process of edge and vertex virtualisation produces two different
graphs $\W^-=(\EG,\Em)$ respectively $\W^+=(\VG,\Vm)$ with induced
vector potentials and weights, such that the spectra of the
corresponding DMLs have the following relation:
\begin{equation*}
  \lambda_k(\Delta_{\EA}^{\EM}) 
  \leq \lambda_k(\Delta_\alpha^\W) 
  \leq \lambda_k(\Delta_{\VA}^{\VM}),
  \quad k=1,\dots, n-r,
\end{equation*}
where $n=|V(\G)|$ and $r$ is the number of virtualised vertices.  We
then say that $\Delta^{\EM}_{\EA}$ is \emph{spectrally smaller} than
$\Delta_\alpha^\W$ (resp.\ $\Delta_\alpha^\W$ is \emph{spectrally
  smaller} than $\Delta_{\VA}^{\VM}$) and write
\begin{equation*}
  \Delta^{\EM}_{\EA}  
  \preccurlyeq \Delta_\alpha^\W
  \preccurlyeq \Delta_{\VA}^{\VM}.
\end{equation*}
Virtualising the edges $E_0$ on which the vector potential is
supported and a corresponding set of vertices $V_0$ in the
neighbourhood of $E_0$ (see Section~\ref{sec:spectral_order} for
precise definitions) we are able to make the DMLs on $\G^\pm$
\emph{independent} of the vector potential $\alpha$.  Hence, also the
bracketing intervals
\begin{equation*}
  J_k
  :=[ \lambda_k(\Delta^{\EM}), \lambda_k(\Delta^{\VM})]
\end{equation*}
in which we are going to localise later the spectrum of the periodic
infinite Laplacian, are independent of the vector potential.

\begin{theorem*}[cf., Theorem~\ref{thm:technique}]
  Let $\W=(\G,m)$ be a finite weighted graph and $E_0\subset E(\G)$.  Then
  for any vector potential $\alpha$ supported in $E_0$ and any vertex
  set $V_0$ in the neighbourhood of $E_0$ we have
  \begin{equation}
    \Delta^{\EM}
    \preccurlyeq \Delta^{\W}_\alpha
    \preccurlyeq \Delta^{\VM},
  \end{equation}
  where $\EM=(\EG,\Em)$ with $\EG=\G-E_0$ and $\VM=(\VG,\Vm)$ with
  $\VG=\G - V_0$.  In particular, we have the spectral localising
  inclusion
  \begin{equation}
    \label{eq:loc.spec}
    \sigma(\Delta^{\W}_\alpha) 
    \subset  J
    := J(\Delta^{\EM},\Delta^{\VM})
    = \bigcup_{k=1}^{|V(\G)|} 
        \bigl[\lambda_k(\Delta^{\EM}),\lambda_k(\Delta^{\VM})\bigr].
   \end{equation}     
\end{theorem*}

In Section~\ref{sec:MSG} we introduce the notion of \emph{magnetic
  spectral gaps set} which in the case of standard weights is given by
\begin{equation*}
  \MSG^\W
  =[0,2] \setminus \UVP \sigma(\Delta^\W_\alpha),
\end{equation*}
where $\mathcal A(\G)$ denotes the set of all vector potentials on
$\G$. In other words, the magnetic spectral gap set is the
intersection of all resolvent sets of the DML $\Delta^\W_\alpha$ for
all $\alpha \in \mathcal A(\G)$, with the set $[0,2]$. If $\G$ is a
tree then $\MSG^\W$ coincides with the spectral gaps set $\SG^\W$ of
the usual Laplacian $\Delta^\W$ with $\alpha=0$. In
Theorem~\ref{TheoremSpectralGaps2} we give a sufficient geometrical
condition on the weighted graph such that the set of magnetic spectral
gaps is non-empty. As a consequence of this result we give several
characterisations of $\MSG^\W\not=\emptyset$ for finite weighted
graphs with Betti number $1$ (cf.,
Corollary~\ref{cor:spectral_gap}). We also present here several
examples of finite graphs with nonempty magnetic spectral gaps set and
show spectral localisation of the spectrum of the DMLs in the
bracketing intervals.

In Section~\ref{sec:periodic} we introduce first basic notation and
results on $\Gamma$-periodic graphs $\Gt$ with finite quotient $\G=\Gt
/\Gamma$ and Abelian discrete group $\Gamma$.  In particular, we
remind a discrete version of Floquet theory and interpret the vector
potential $\alpha$ on $\G$ as a Floquet parameter. Applying then
results of the previous sections to the finite quotient we prove of
Higuchi-Shirai's conjecture for $\Z$-periodic trees, i.e., we prove
the following result:
 \begin{theorem*}[cf., Theorem~\ref{theo:FSP}]
   Let $\Wt =(\Gt,\m)$ be a $\Z$-periodic tree with standard or
   combinatorial weights and quotient graph $\W=(\G,m)$.  Then the
   following conditions are equivalent:
   \begin{enumerate}
   \item $\Wt$ has the full spectrum property;
   \item $\SG^{\Wt}=\emptyset$;
   \item $\Wt$ is the lattice $\Z$;
   \item $\MSG^{\W}=\emptyset$;
   \item $\G$ is a cycle graph;
   \item $\G$ has no vertex of degree $1$.
   \end{enumerate}
 \end{theorem*}
 Finally, we also apply the bracketing method to show that the
 spectrum of the Laplacian $\Delta^{\Gt}$ is localised in the union of
 the bracketing intervals. This gives a sufficient condition for the
 existence of spectral gaps in the spectrum of $\Delta^{\Gt}$.

In Section~\ref{sec:examples} we apply the methods developed to
several classes of examples of periodic graphs.  The first example
gives simple verification of results by Suzuki in~\cite{suzuki:13} on
the existence of spectral gaps for $\Z$-periodic graphs with pendants.
The other examples are more elaborate and include modelisations of
polypropylene and polyacetylene molecules. We prove also the existence
of spectral gaps in these periodic structures. The last example can be
understood as an intermediate covering of the graphene where the
quotient has Betti number $2$.

\subsection*{Notation}
We denote a weighted graph by $\W=(\G,m)$ where $\G=( V, E, \bd )$ is
an oriented graph and $m$ a weight on the vertices $V$ and edges
$E$. For a vector potential $\alpha$, the operator $\Delta_\alpha^\G$
denotes the discrete magnetic Laplacian (DML) with standard weights
and $\Delta_\alpha^\W$ corresponds to the DML to denote a generic
weighted graph. We use $\Delta^\G$, $\MSG^\G$ and $\SG^\G$ to denote
the usual Laplacian, the set of magnetic spectral gaps and the set of
spectral gaps with standard weights, respectively, and $\Delta^\W$,
$\MSG^\W$ and $\SG^\W$ for the corresponding objects with generic
weights. We denote $\Wt=(\Gt,\m)$ a periodic weighted graph.

\subsection*{Acknowledgements}
We would like to thank Pavel Exner for encouraging us to extend the
previous article on bracketing techniques in~\cite{lledo-post:08b}
also to \emph{magnetic} Laplacians.  We would also like to thank the
anonymous referee for carefully reading our manuscript and for his
useful suggestions.

%
%
\section{Discrete graphs and discrete magnetic Laplacians}
\label{sec:DML}
%

In this section we introduce the necessary background on discrete graphs and Laplacians
that will be used later.

\subsection{Graphs and subgraphs}
\label{ssec:graphs}

In this article $\G=( V, E, \bd)$ denotes a (discrete) directed graph,
i.e., $V=V(\G)$ is the set of vertices, $E=E(\G)$ the set of edges and
$\bd \colon E\rightarrow V\times V$ is the orientation map.  Here,
$\bd e=(\bd_-e,\bd_+e)$ denotes the pair of the initial
and terminal vertices.  We allow graphs with multiple edges, i.e.,
edges $e_1\neq e_2$ with $(\bd_-e_1,\bd_+e_1 ) =(
  \bd_-e_2,\bd_+e_2 )$ or $(\bd_-e_1,\bd_+e_1 )
=( \bd_+e_2,\bd_-e_2 )$ and loops, i.e., edges $e_1$ with
$\bd_-{e_1}=\bd_+{e_1}$. We define
\begin{equation*}
  E_v:=E_v^+ \dcup E_v^- \quadtext{(disjoint union), where}
  E^\pm_v :=\lbrace e\in E \mid v=\bd_\pm e\rbrace  \;.
\end{equation*}
The degree of a vertex is $\deg(v)=|E_v|$; note that, since $E_v$ is defined in terms of a disjoint union,
a loop increases the degree by $2$. 

For subsets $A,B \subset V$, denote by
\begin{equation*}
  E^+(A,B) := \lbrace e \in E \mid \bd_-e \in A, \bd_+e \in B\rbrace
  \quadtext{and}
  E^-(A,B) := E^+(B,A).
\end{equation*}
In particular, $E^\pm_v = E^\pm(V,\{v\})$.  Moreover, we set
\begin{equation*}
  E(A,B) := E^+(A,B) \cup E^-(A,B) 
  \qquadtext{and}
  E(A):= E(A,A).
\end{equation*}
To simplify the notation, we write $E(v,w)$ instead of
$E(\{v\},\{w\})$ etc.  Note that loops are not counted double in
$E(A,B)$, in particular, $E(v):=E(v,v)$ is the set of loops based at
the vertex $v\in V$.
The \emph{Betti number} $\Betti(\G)$ of a finite graph $\G=( V, E,
  \bd)$ is defined as
\begin{equation}
  \label{eq:betti}
  \Betti(\G):=|E|-|V|+1.
\end{equation}

When we analyse in the next sections virtualisation processes of
vertices, edges and fundamental domains in periodic graphs, it will be
convenient to consider the following substructure of a graph.

\begin{definition}
  \label{def:subgraphs}
  Let $\G=(V,E,\bd)$ be a discrete graph and
  $\subG=(V_0,E_0,\bd_0)$ be a triple such that $V_0 \subset V$, $E_0
  \subset E$ and $\bd_0 =\bd \restriction_{E_0}$.
  \begin{enumerate}
  \item 
    \label{subgraphs.a}
    If $E_0 \cap E(V \setminus V_0) = \emptyset$, we say that $\subG$
    is a \emph{partial subgraph} in $\G$.  We call
    \begin{align}
      \nonumber
      B(\subG,\G) 
      :=& E(V_0,V \setminus V_0)\\
      \label{eq:def.bridges}
      =&\lbrace e \in E 
        \mid \bd_-e \in V_0, \bd_+e \in V \setminus V_0 \text{ or }
             \bd_+e \in V_0, \bd_-e \in V \setminus V_0
      \rbrace
    \end{align}
    the set of \emph{connecting edges} of the partial subgraph $\subG$
    in $\G$.
  \item 
    \label{subgraphs.b}
    If $E_0 \subset E(V_0)$, we say that $\subG$ is a \emph{subgraph} of $\G$. 

  \item 
    \label{subgraphs.c}
    If $E_0 = E(V_0)$, we say that $\subG$ is an \emph{induced
      subgraph} of $\G$.
	
  \item 
      \label{subgraphs.d}
      If $\subG$ is a \emph{subgraph} of $\G$ with $V_0=V(\G)$ such that
      $\subG$ is connected and has no cycles, we say that $\subG$ is
      an \emph{induced tree} of $\G$.
  \end{enumerate}
\end{definition}
\begin{remark}
  \indent
  \begin{enumerate}
  \item Note that a partial subgraph $\subG=(V_0,E_0,\bd_0)$ is not a
    graph as defined in Section~\ref{ssec:graphs}, since we do not
    exclude edges $e \in E$ with $\bd_\pm e \notin V_0$; we only
    exclude the case that $\bd_+e \notin V_0$ \emph{and} $\bd_-e
    \notin V_0$.  The edges not mapped into $V_0 \times V_0$ under
    $\bd_0$ are precisely the connecting edges of $\subG$ in $\G$.  In
    other words, we only have $\bd_0 \colon B(\subG,\G) \rightarrow V
    \times V$, but $\bd_0 \colon E_0\setminus B(\subG,\G) \rightarrow
    V_0 \times V_0$.
  \item In contrast, for a subgraph or an induced subgraph, $\subG$ is
    itself a discrete graph as $E_0 \subset E(V_0)$ and hence $\bd_0$
    maps $E_0$ into $V_0 \times V_0$.  Moreover, a subgraph (or an
    induced subgraph) has no connecting edges in $\G$.
  \end{enumerate}
\end{remark}

\subsection{Weighted discrete graphs}
\label{ssec:weight}

Let $\G=(V,E,\bd)$ be a discrete graph.  A \emph{weight} on $\G$ is a
pair of functions $m$ on the vertices and edges
$m\colon V\rightarrow (0,\infty)$ and $m\colon E\rightarrow (0,\infty)$
associating to a vertex $v$ its weight $m(v)$ and to an edge $e$ its weight 
$m_e$.\footnote{Later, the virtualization process of edges can also be 
interpreted allowing $m_e=0$ on certain edges $e$.}

We call $\W=(\G,m)$ a weighted discrete graph.  Now, it is natural to
define $m(E_0)=\sum_{e \in E_0} m_e$ for any $E_0 \subset E$.  The
\emph{relative weight} is $\rho \colon V\rightarrow (0,\infty)$
defined as
\begin{subequations}
  \label{eq:rel.weight}
  \begin{equation}
    \label{eq:rel.weight.a}
    \rho(v)
    :=\dfrac{m(E_v)}{m(v)}
    =\dfrac{m(E_v^+)+m(E_v^-)}{m(v)}.
  \end{equation}
  If we need to stress the dependence of $\rho$ of the weighted graph,
  we simply write $\rho^\W$.  We will assume throughout this article
  that the relative weight is uniformly bounded, i.e.,
  \begin{equation}
    \label{eq:rel.weight.b}
    \rho_\infty:=\sup_{v\in V} \rho(v)<\infty.
  \end{equation}
\end{subequations}
This condition will ensure later on that the discrete magnetic
Laplacian is a bounded operator.

Examples of commonly used weights are the following:
\begin{center}
  \begin{tabular}{|l|c|c|c|c|}
    \hline 
    \textbf{Weight name}  &  $m_e$& $m(v)$ &$\rho(v)$&$\rho_\infty$\\ 
    \hline 
    \emph{standard}\index{Standard weight} & $1$ & $\deg v$&$1$ &$1$ \\ 
    \hline \emph{combinatorial}\index{Combinatorial weight}& $1$ &  $1$ 
    &$\deg v$ &$\sup_{v\in V} \deg v$ \\[1ex] 
    \hline \emph{normalised}\index{Normalised weight} & $m_e$ & 
    $ m (E_v)$ & $1$ &$1$ \\[1ex] 
    \hline  \emph{electric circuit}\index{Electric Circuit} & $m_e$ &  $1$& 
    $ m (E_v)$&  
    $\sup_{v\in V} m (E_v)$\\
    \hline
  \end{tabular} 
\end{center}
Note that the first two weights are \emph{intrinsic}, i.e., they can
be calculated just by the graph data, while the last two need the
additional information of an edge weight $m \colon E \to (0,\infty)$.

To a weighted graph $\W=(\G,m)$ we associate the following two
Hilbert spaces
\begin{align*}
  \ell_2(V,m)
  &:=\Bigl\lbrace f\colon V\rightarrow \C \mid
    \left\| f \right\|_{V,m}^2=\sum_{v \in V} | f(v)|^2 m(v) < \infty  \Bigr\rbrace
  \qquad\text{and}\\
  \ell_2(E,m)
  &:=\Bigl\lbrace \eta\colon E\rightarrow \C \mid
    \left\| \eta \right\|_{E,m}^2=\sum_{e \in E} | \eta_e|^2 m_e < \infty 
      \Bigr\rbrace,
\end{align*}
with inner products
\begin{equation*}
  \left\langle f,g\right\rangle_{\ell_2(V,m)}
  =\sum_{v \in V} {f(v)}  \overline{g(v)}  m(v)
  \quadtext{and} 
  \left\langle \eta,\zeta \right\rangle_{\ell_2(E,m)}
  =\sum_{e \in E} \eta_e  \overline{\zeta_e}  m_e,
\end{equation*}
respectively.  These spaces can be interpreted as $0$- and $1$-forms
on the graph, respectively.

\subsection{Discrete magnetic Laplacian}\label{ssec:DML}

Let $\W=(\G,m)$ a weighted graph. A \emph{vector potential} $\alpha$
acting on $\G$ is a $\Torus$-valued function on the edges as follows, $\alpha\colon
E(\G)\rightarrow \Torus=\R/2\pi\Z.$ We denote the set of all vector 
potentials on $E(\G)$ just by $\mathcal{A}(\G)$.
We say that two vector potentials
$\alpha_1$ and $\alpha_2$ are \emph{cohomologous}, and denote this as
$\alpha_1 \sim \alpha_2$, if there is $\varphi\colon V\rightarrow
\Torus$ with
\begin{equation*}
  \alpha_1 = \alpha_2 + d\varphi.
\end{equation*}
Given a $E_0 \subset E(\G)$, we say that a vector potential $\alpha$
has support in $E_0$ if $\alpha_e=0$ for all $e\in E(\G)\setminus
E_0$. 

It can be shown that any vector potential on a finite graph can be
supported in $\Betti(\G)$ many edges. In fact, let $\G$ a finite graph
with $\alpha$ a vector potential acting on it and let $\T$ an induced
tree of $\G$. Then we can show that there exists a vector potential
$\alpha'$ with support in $E(\G)\setminus E(\T)$ such that $\alpha\sim
\alpha'$. In particular, if $\G$ is a cycle, any vector potential is
cohomologous to a vector potential supported in only one edge.
Moreover, if $\G$ is a tree any vector potential on a tree is
cohomologous to $0$.

The \emph{twisted (discrete) derivative} is the operator between
$0$-forms and $1$-forms given by
\begin{equation}
  \label{eq:twist.der}
  d_\alpha\colon \ell_2(V,m) \rightarrow \ell_2(E,m) 
  \qquadtext{with}
  \left( d_\alpha f \right)_e
  =\e^{\im \alpha_e/2}f(\bd_+e) - \e^{-\im \alpha_e/2} f(\bd_-e).
\end{equation}  

\begin{definition}
  Let $\W=(\G,m)$ be a weighted graph with $\alpha$ a vector
  potential. The \emph{discrete magnetic Laplacian (DML)}
  $\Delta_\alpha  \colon
  \ell_2(V)\rightarrow \ell_2(V)$ is defined by
  $\Delta_\alpha=d_\alpha^*d_\alpha$, i.e., by
  \begin{equation*}
    \left( \Delta_\alpha f\right) \left(v \right) 
    =\rho(v)f(v)-\dfrac1{m(v)}\sum_{e\in E_v}  {\e^{\im \orient \alpha_e(v)}}f(v_e) m_e,
  \end{equation*}
  where $\orient \alpha_e(v)$ resp.\ $v_e$ is the oriented evaluation
  resp.\ opposite vertex of $v$ along the edge $e$, i.e.,
  \begin{equation*}
    \orient \alpha_e(v)
    = 
    \begin{cases}
      -\alpha_e,  & \text{if $v=\bd_-e$,}\\
      \alpha_e,  & \text{if $v=\bd_+e$,}
    \end{cases}
    \quadtext{resp.}
    v_e
    = 
    \begin{cases}
      \bd_+e,  & \text{if $v=\bd_-e$,}\\
      \bd_-e   & \text{if $v=\bd_+e$.}
    \end{cases}
  \end{equation*}
  If we need to stress the dependence on the weighted graph
  $\W=(\G,m)$ we will denote the DML as $\Delta_\alpha^\W$.
\end{definition}

The DML $\Delta_\alpha$ is a bounded, positive and self-adjoint
operator and its spectrum satisfies $\sigma(\Delta_\alpha)\subset
[0,2\rho_\infty]$.  Unlike the usual Laplacian without magnetic
potential, the DML \emph{does} depend on the orientation of the graph.
If $\alpha\sim \alpha'$, then $\Delta_{\alpha}$ and $\Delta_{\alpha'}$
are unitary equivalent; in particular,
$\sigma(\Delta_\alpha)=\sigma(\Delta_{\alpha'})$. Indeed, if
$\alpha'=\alpha+d\varphi$ then it is straightforward to check that the
unitary operator $(Uf)(v)=\e^{\varphi(v)}f(v)$ intertwines between
both DMLs.  In particular, if $\alpha\sim 0$ then $\Delta_\alpha\cong
\Delta$ where $\Delta$ denotes the discrete Laplacian with vector
potential $0$, i.e., the usual discrete Laplacian on $(\G,m)$.  For
example, if $\W=(\G,m)$ with $\G$ being a tree, then
$\Delta_\alpha^{\W} \cong \Delta^{\W}$ for any vector potential.

If the graph $G=(V,E,\bd)$ is bipartite (i.e., there is a partition
$V=A \dcup B$ such that $E=E(A,B)$, we have the following spectral
symmetry (see, e.g.,~\cite[Prp.~2.3]{lledo-post:08}):
\begin{proposition}
  \label{prp:bipartite.sym}
  Assume that $\W=(\G,m)$ is a weighted graph with bipartite graph
  $\G$ and normalised weight $m$.  Then the spectrum of $\Delta^\W$
  (with vector potential $\alpha=0$) is symmetric with respect to the
  map $\kappa \colon \R \to \R$, $\kappa(\lambda)=2-\lambda$, i.e.,
  \begin{equation*}
    \kappa(\sigma(\Delta^\W))=\sigma(\Delta^\W).
  \end{equation*}
  In particular, if $J \subset [0,2]$ fulfils $\sigma(\Delta^\W)
  \subset J$, then we have the inclusion
  \begin{equation*}
    \sigma(\Delta^\W) \subset J \cap \kappa(J).
  \end{equation*}
\end{proposition}
Note that the set $J \cap \kappa(J)$ becomes smaller than $J$ if $J$
is \emph{not} symmetric with respect to $\kappa$.

\subsection{Matrix representation of the DML}
\label{MatrixDML}

For computing the eigenvalues of the \emph{DML} it is convenient to work
with the associated matrix.  Given a finite weighted graph $(\G,m)$
with vector potential $\alpha$.  Consider a numbering of the vertices as
$V(\G)=\left\lbrace v_1,v_2, \dots ,v_n\right\rbrace $.  Then
$\left\lbrace \varphi_{v_i} \right\rbrace_{i=1}^n \subset
\ell_2(V,m)$ with $\varphi_{v_i} =m(v_i)^{-1/2} \1_{\left\lbrace
    v_i\right\rbrace}$ is an orthonormal basis of $\ell_2(V,m)$.  The
matrix representation of $\Delta_\alpha$ with respect to this
orthonormal basis is given by

\begin{equation*}
  \left[ \Delta_\alpha\right] _{jk}
  =\begin{cases} 
    \displaystyle \rho(v_j)-\frac 1 {m(v_j)} \sum_{e\in E(v_j,v_j)} 
     \Bigl(\underbrace{\e^{\im \alpha_e}+\e^{- \im \alpha_e}}_{=2\cos \alpha_e}
     \Bigr)  m_e, & \text{if $j=k$,}\\[2ex]
  \displaystyle -\frac 1 {\sqrt{m(v_j)m(v_k)}}  
     \Bigl(\sum_{e\in E^+(v_j,v_k)} \e^{\im  \alpha_e} m_e 
        +  \sum_{e\in E^-(v_j,v_k)} \e^{-\im  \alpha_e} m_e  
     \Bigr), & \text{if $ v_j \sim v_k$,}  \\[2ex]
    0 & \text{otherwise.}
  \end{cases}
\end{equation*}
where $v_j \sim v_k$ meaning that $v_j$ and $v_k$ are connected by an edge. Note that this formula includes the case of graphs with multiple edges and loops.

%
\section{Spectral ordering on finite graphs}
\label{sec:spectral_order}
%

In this section we will introduce one spectral ordering and two
operation on the graphs that will be needed later to develop a
discrete bracketing technique and show the existence of spectral gaps
for Laplacians on periodic graphs. Korotyaev and Saburova also present
a discrete bracketing technique in~\cite{koro-sabu:15} for
combinatorial weights. Their Dirichlet upper bound of the bracketing
is similar to the one we use here (vertex-virtualised). For the lower
bound Korotyaev and Saburova use a Neumann type boundary condition
while we propose an alternative edge virtualization process using the fact that we
work with arbitrary weights. In our approach the virtualisation is
done is such a way that vector potential on the resulting graphs
(deleting edges and deleting vertices) is cohomologous to zero.

Let $\W=(\G,m)$ a weighted graph.  Throughout this section, we will
assume that $|V(\G)|<\infty$.  If $|V(\G)|=n$, then we denote the spectrum
of the \emph{DML}  by $\sigma(\Delta^\W_\alpha):=\{ \lambda_k(\Delta^\W_\alpha) \mid
k=1,\dots,n\}$, where we will write the eigenvalues in ascending order
and repeated according to their multiplicities, i.e.,
\begin{equation*}
  0 \leq 
  \lambda_1(\Delta^\W_\alpha) 
  \leq \lambda_2(\Delta^\W_\alpha)
  \leq  \cdots \leq \lambda_n(\Delta^\W_\alpha).
\end{equation*}

\begin{definition}\label{def:PO}
  Let $S^-$ and $S^+$ be self-adjoint operators on $n^-$- respectively
  $n^+$-dimensional Hilbert spaces and consider the eigenvalues written
  in ascending order and repeated according to their
  multiplicities. We say that $S^-$ is \emph{spectrally smaller} than
  $S^+$ (denoted by $S^- \preccurlyeq S^+$), if
  \begin{equation*}
    n^- \ge n^+
    \qquadtext{and if} 
    \lambda_k(S^-)\leq \lambda_k (S^+) 
    \quadtext{for all}
    1\leq k \leq n^+.
  \end{equation*}
\end{definition}
Assume now that all operators have spectrum (i.e., eigenvalues) in
$[0,2\rho_\infty]$ for some number $\rho_\infty>0$.  If we set
$\lambda_k(S^+)=2\rho_\infty$ for $k= n^++1, \dots, n^-$ (the maximal
possible eigenvalue), then we can replace the eigenvalue estimate in
the previous definition by
\begin{equation*}
  \lambda_k(S^-)\leq \lambda_k (S^+) 
  \quadtext{for all}
  1\leq k \leq n^-.
\end{equation*}
Note also that the relation $\preccurlyeq$ is invariant under unitary conjugation of the operator.
\begin{definition}
  \label{def:brack.int}
  For operators $S^-$ and $S^+$ with $S^- \preccurlyeq S^+$ we define
  the \emph{associated $k$-th bracketing interval} $J_k=J_k(S^-,S^+)$
  by
  \begin{equation}
    \label{eq:brack.int}
    J_k := \bigl[\lambda_k(S^-),\lambda_k(S^+)\bigr]
  \end{equation}
  for $k=1,\dots, n^-$.
\end{definition}

If now $T$ is an operator with $S^- \preccurlyeq T \preccurlyeq S^+$,
then we have the following \emph{eigenvalue bracketing}
\begin{equation}
  \label{eq:ev.brack}
  \lambda_k(T) \in J_k
\end{equation}
for all $k=1,\dots, n^-$.  Moreover,
\begin{equation}
  \label{eq:ev.brack.all}
  \sigma(T) \subset J:= \bigcup_{k=1}^{n^-} J_k
\end{equation}
and we call $J=J(S^-,S^+)$ the \emph{spectral localising set} of the
pair $S^-$ and $S^+$.

The key observation for detecting spectral gaps from the eigenvalue
bracketing is the following:
\begin{proposition}
  \label{prp:key-obs}
  Let $S^-$ and $S^+$ be operators on $n^-$, respectively $n^+$, 
  finite dimensional Hilbert spaces
  with spectrum in $[0,2\rho_\infty]$.  Let $\mathcal T$ be a subset of the set
  of operators $T$ on a finite dimensional Hilbert spaces with $S^-
  \preccurlyeq T \preccurlyeq S^+$, then
    \begin{equation*}
      \mu\Bigl( \bigcup_{T \in \mathcal T} \sigma(T)
      \Bigr) \le \Tr(S^+)-\Tr(S^-) + 2\rho_\infty(n^--n^+)
    \end{equation*}
where $\mu$ denotes the $1$-dimensional Lebesgue measure.  In
    particular, if $n^--n^+=1$ and $\Tr(S^+)-\Tr(S^-)<0$, then $\bigcup_{T \in
      \mathcal T} \sigma(T)$ cannot be the entire interval
    $[0,2\rho_\infty]$.
\end{proposition}
\begin{proof}
  We have
  \begin{equation*}
    \mu\Bigl( \bigcup_{T \in \mathcal T} \sigma(T)\Bigr) 
    \le  \mu\Bigl(  \bigcup_{k=1}^{n^-} J_k\Bigr) 
    \le  \sum_{k=1}^{n^-}\bigl( \lambda_k(S^+)-\lambda_k(S^-))
    = \Tr(S^+) + 2\rho_\infty(n^--n^+)  -\Tr(S^-).
    \qedhere
  \end{equation*}
\end{proof}

We begin with the description of a procedure of manipulation of the
graph that will lead to a spectrally smaller DML.

\begin{definition}[\emph{virtualising} edges]
  \label{deleteedges}
  Let $\W=(\G,m)$ be a weighted graph with vector potential $\alpha$ and
  $E_0\subset E(\G)$.  We denote by $\EM=(\EG,\Em)$ the weighted
  subgraph with vector potential $\EA$ defined as follows:
  \begin{enumerate}
  \item $V(\EG)=V(\G)$ with $\Em(v):=m(v)$ for all $v\in V(\G)$;
  \item $E(\EG)=E(\G)\setminus E_0$ with $\Em_e:=m_e$ and
    $\bd_{\pm}^{\EG} e=\bd_{\pm}^\G e$ for all $e\in E(\EG)$;
  \item $\EA_e=\alpha_e$, $e\in E(\EG)$.
  \end{enumerate} 
  We call $\EM$ the weighted subgraph obtained from $\W$ by
  \emph{virtualising the edges $E_0$}. We will sometimes also use the
  suggestive notation $\EG=\G-E_0$.

  The corresponding discrete magnetic Laplacian is denoted by
  $\Delta_{\EA}^{\EM}$.
\end{definition}

\begin{remark}
  \label{rem:VE}
  \indent
  \begin{enumerate}
  \item 
    Note that $\Delta_{\EA}^{\EM}=(d_{\alpha^-})^*d_{\alpha^-}$, where
    $d_{\alpha^-}:= \pi \circ d_\alpha$ and $\pi \colon \ell_2(E(\G),m)
    \rightarrow \ell_2(E(\EG),\Em)$ is the orthogonal projection onto
    the functions on the non-virtualised edges. Note that
    $\pi=\iota^*$ with $\iota \colon \ell_2(E(\EG),\Em) \rightarrow
    \ell_2(E(\G),m)$ being the natural inclusion, i.e., for $\eta\in
    \ell_2(E(\EG),\Em)$ $\iota\eta$ is extended by $0$ on
    $E(\EG)=E(\G)\setminus E_0$.  Note that the process of
    virtualisation of edges can be also described by changing the
    weights on $\G$: set $m_{e_0}=0$ for $e_0\in E_0$, and leave all
    other weights unchanged.
  \item
    \label{VE.a}
    The process of \emph{virtualising} edges has consequences for
    various quantities related to the graph. The most important for us
    here refers to the spectrum (see the following proposition). If
    $\rho^-$ denotes the relative weight of $\EM$, then $\rho^-(v)\leq
    \rho(v)$ for $v \in V$.

    Let $\W=(\G,m)$ be a weighted graph with standard weights, and
    denote by $\W'=(\EG,m')$ the graph $\G^-$ with standard weights.  
    If $E_0 \neq \emptyset$, then there exists a
    $v\in V(\G)$ such that $\Em(v) > m'(v)$, i.e., the new weight
    $\Em$ is not standard anymore.
    More generally, if $\W=(\G,m)$ has a normalised weight, then $\Em$
    is no longer normalised; the relative weights of $\EM$ and $\W'$
    fulfil $\rho^-(v) \le \rho'(v)=1$ with strict inequality for $v
    \in V(\EG)$ incident with an edge in $E_0$.

    If $\W=(\G,m)$ is a weighted graph with combinatorial weight
    $m$, then $\Em$ as well has the combinatorial weight
    , i.e., combinatorial weights are
    preserved under edge virtualisation.

  \item It is also clear that if $\G$ is connected, then $\EG$ need
    not to be connected any more.  Moreover, the homology of the graph
    changes under edge virtualisation.  This is perhaps an
    important motivation of this definition.  Later we will exploit
    the fact that deleting a suitable set of edges the graph will turn it
    into a tree. The graph $\G^-$ is sometimes also called spanning subgraph.
  \end{enumerate}
\end{remark}

We show next that the process of virtualising edges produces a DML
which is spectrally smaller (cf., Definition~\ref{def:PO}).

\begin{proposition}
  \label{VirtualEdges}
  Let $\W=(\G,m)$ be a weighted graph with vector potential $\alpha$ and
  $E_0\subset E(\G)$. Denote by $\EM=(\EG,\Em)$ with $\EG=\G-E_0$
  the edge virtualised graph (cf.\ Definition~\ref{deleteedges}), then
  \begin{equation*}
    \Delta_{\EA}^{\EM} \preccurlyeq \Delta_\alpha^\W.
  \end{equation*}
\end{proposition}
\begin{proof}
  Since $E(\EG) \subset E(\G)$ we have for $f\in \ell_2(V(\G),m)$  
  \begin{align*}
    \iprod {\Delta_{\EA}^{\EM} f} f
    = \normsqr[\ell_2(E(\EG),\Em)] {d_{\alpha^-} f}
    \le \normsqr[\ell_2(E(\G),m)] {d_\alpha f}
    =  \iprod {\Delta_\alpha^\W f} f \;.
  \end{align*}
    For $k \in\left\lbrace 1,2,\dots,|V(\G)|\right\rbrace $ denote by $\Xi_k$ the
  set of all intersections $E_k$ of $k$-dimensional subspaces with the
  $1$-sphere in $\ell_2 (V(\G),m)$.  Applying the variational
  characterisation of the spectrum (see,
  e.g.,~\cite[Theorem~6.1.2]{blanchard2012variational}) we conclude
  \begin{equation*}
    \lambda_k(\Delta_{\EA}^{\EM} )
    =\min_{E_k\in \Xi_k} \max_{f\in E_k} \normsqr {d_{\alpha^-} f} 
    \leq \min_{E_k\in \Xi_k} \max_{f\in E_k} \normsqr {d_\alpha f}. 
  \end{equation*}
  Since $|V(\G)|= |V(\EG)|$ we have shown $\Delta_{\EA}^{\EM} \preccurlyeq \Delta_\alpha^\W$.
\end{proof}

\begin{example}
  \label{exampledeleteedge}
   Let $\W=(\G,m)$ be the $6$-cycle with
  standard weights as in Figure~\ref{fig:smaller}.  Let
  $\EM=(\EG,\Em)$ be the graph with the edge $e_1$ virtualised (i.e.,
  $\EG=\G-\{e_1\}$, see Definition~\ref{def:PO}, and $\Em$ being the
  restriction of $m$ to $E(\EG)=E(\G) \setminus \{e_1\}$.  If $\alpha$
  is any vector potential on $\G$, then $\alpha$ can be supported on
  $e_1$, so $\EA$ is trivial on $\EG$. Therefore
  $\Delta_{\EA}^{\EM}$ is unitarily equivalent with $\Delta^{\EM}$ (usual Laplacian 
  with $\alpha=0$). Finally, we plot the
  six eigenvalues of $\Delta_\alpha^{\G}$ and $\Delta^{\EM}$, when
  $\alpha_{e_1}$ runs through $[0,2\pi]$. This example illustrates
  $\Delta_{\EA}^{\EM} \preccurlyeq \Delta_\alpha^\G$ hence, by unitary equivalence, also 
  $\Delta^{\EM} \preccurlyeq \Delta_\alpha^\G$. 

  \begin{figure}[h] \label{subfig:1}
      \centering \subcaptionbox{The graph $\G=C_6$.}
      [.3\linewidth]{\begin{tikzpicture}[auto,
          vertex/.style={circle,draw=black!100,fill=black!100, thick,
            inner sep=0pt,minimum size=1mm}] \node (A) at (1/2,-2/2)
          [vertex,label=below:$v_5$] {}; \node (B) at (-1/2,-2/2)
          [vertex,label=below:$v_6$] {}; \node (C) at (-2/2,0)
          [vertex,label=left:$v_1$] {}; \node (D) at (-1/2,2/2)
          [vertex,label=above:$v_2$] {}; \node (E) at (1/2,2/2)
          [vertex,label=above:$v_3$] {}; \node (F) at (2/2,0)
          [vertex,label=right:$v_4$] {};
      
  	\path [-latex](B) edge node[right] {} (C);
  	\path [-latex](C) edge node[right] {$e_1$} (D);
  	\path [-latex](D) edge node[below] {} (E);
  	\path [-latex](E) edge node[left] {} (F);
  	\path [-latex](F) edge node[left] {} (A);
  	\path [-latex](A) edge node[above] {} (B);
      \end{tikzpicture}} \subcaptionbox{$\EG=C_6 - \{e_1\}$.}  [.3\linewidth]{\begin{tikzpicture}[auto,
        vertex/.style={circle,draw=black!100,fill=black!100, thick,
          inner sep=0pt,minimum size=1mm}] \node (A) at (1/2,-2/2)
        [vertex,label=below:$v_5$] {}; \node (B) at (-1/2,-2/2)
        [vertex,label=below:$v_6$] {}; \node (C) at (-2/2,0)
        [vertex,label=left:$v_1$] {}; \node (D) at (-1/2,2/2)
        [vertex,label=above:$v_2$] {}; \node (E) at (1/2,2/2)
        [vertex,label=above:$v_3$] {}; \node (F) at (2/2,0)
        [vertex,label=right:$v_4$] {};
      
  	\path [-latex](B) edge node[right] {} (C);
  	\path [-latex](D) edge node[below] {} (E);
  	\path [-latex](E) edge node[left] {} (F);
  	\path [-latex](F) edge node[left] {} (A);
  	\path [-latex](A) edge node[above] {} (B);
      \end{tikzpicture}} \subcaptionbox{Spectra of
      $\Delta_\alpha^{C_6}$ (black lines) and
      $\Delta_{\EA}^{\EG}$ (dashed lines).}
    [.3\linewidth]{\includegraphics[width=.3\linewidth]{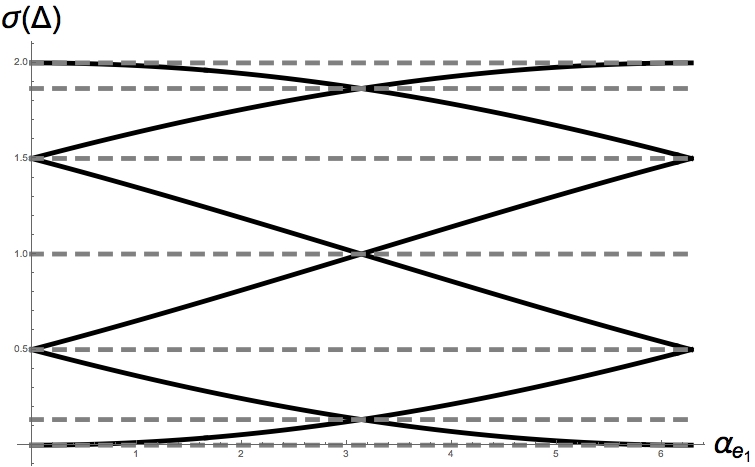}}
    \caption{Virtualization of an edge of $C_6$.\label{fig:smaller}}
\end{figure}
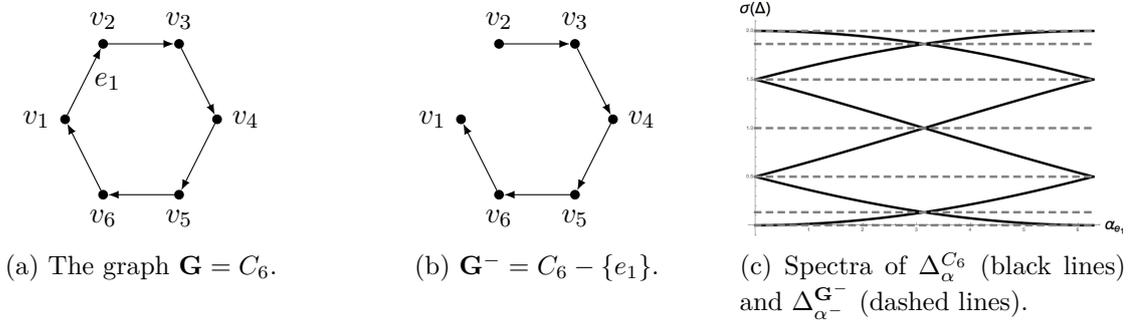
\end{example}

\begin{example}
  This example shows that Proposition~\ref{VirtualEdges} is not true
  if we insist on having standard weights both in $\G$ and $\EG$ (cf.,
  Remark~\ref{rem:VE}~\eqref{VE.a}). Take $\G$ and $\G^{-}=C_6 - \{e_1\}$
  both with the standard weights. Consider the vector potential
  supported on $e_1$ such that $\alpha_{e_1}=\pi/2$.  In this case
  we have $\lambda_{4}(\Delta^{\G^{-}}_{\EA}) =
  \lambda_{4}(\Delta^{\G^{-}}) \approx 1.30902 \nleq 1.25882 \approx
  \lambda_{4}(\Delta^{\G}_{\alpha})$,  hence
  $\Delta_{\EA}^{\G^{-}}$ is not spectrally smaller than $
  \Delta_\alpha^{\G}$.
\end{example}

We next describe the second elementary operation on the graph dual to
edge virtualising.
\begin{definition}[\emph{virtualising} vertices]
  \label{deletevertices}
  Let $\W=(\G,m)$ be a weighted graph with vector potential $\alpha$ and
  $V_0\subset V(\G)$.  We denote by $\VM=(\VG,\Vm)$ the weighted
  partial subgraph with vector potential $\VA$ defined as
  follows:
  \begin{enumerate}
  \item $V(\VG)=V(\G) \setminus V_0$ with $\Vm(v):=m(v)$ for all $v\in V(\VG)$;
  \item $E(\VG)=E(\G) \setminus E(
  V_0)$ with
    $\Vm_e:=m_e$ for all $e\in E(\VG)$;
  \item $\VA_e=\alpha_e$, $e\in E(\VG)$.  
  \end{enumerate} 
  We call $\VM$ the weighted partial subgraph obtained from $\W$ by
  \emph{virtualising the vertices $V_0$}. We will also use the
  suggestive notation $\VG=\G-V_0$.

  The corresponding discrete magnetic Laplacian is defined by 
  \begin{equation*}
    \Delta_{\VA}^{\VM}=(d_{\VA})^*d_{\VA},
    \qquadtext{where}
    d_{\VA}:= d_\alpha \circ \iota
  \end{equation*}
  with
  \begin{equation*}
    \iota \colon \ell_2(V(\VG),\Vm) \rightarrow \ell_2(V(\G),m),
    \qquad
    (\iota f)(v)=
    \begin{cases}
      f(v), &  v \in V(\VG),\\
      0, & v \in V_0.
    \end{cases}
  \end{equation*}
\end{definition}

\begin{remark}
  \label{rem:VV}
  \indent
  \begin{enumerate}
  \item
    \label{VV.0}
    Here, we use for the first time the notion of \emph{partial
      subgraphs} having edges with only one vertex in the $V(\VG)$,
    the other one being in $V_0$; one can also call the vertices in
    $V_0$ \emph{virtual}.  Formally, $\bd \colon E(\VG) \to V(\G)
    \times V(\G)$ still maps into the product of the vertex set, but
    of the \emph{original} graph, not into $V(\VG) \times V(\VG)$.

    In general, $(V(\VG),E(\VG),\bd)$ is not a graph in the classical
    sense anymore, as some edges have initial or terminal vertices no
    longer in $V(\VG)$.  One can actually see that there is no such
    proper weighted graph $\W_1=(V(\VG),E_1,\bd)$. In fact, the
    corresponding Laplacian $\Delta^{\W_1}$ has $0$ as lowest
    eigenvalue, but $\Delta^{\VM}$ does not have $0$ as lowest
    eigenvalue (provided $V_0 \ne \emptyset$) since any function with
    $\Delta^{\VM} f=0$ has to be constant on $V$, and $0$ on $V_0$,
    hence $f=0$.
  \item
    \label{VV.a}
    The definition of $d_{\VA}=d_\alpha \circ \iota$ is consistent
    with the natural definition of $d_{\VA}$ for a partial subgraph,
    namely we set
    \begin{equation*}
      (d_{\VA} f)_e=
      \begin{cases}
        \e^{\im \alpha_e/2}f(\bd_+e)-\e^{-\im \alpha_e/2}f(\bd_-e),
             &\text{if $\bd_\pm e \in V(\VG)$,}\\
        \e^{\im \alpha_e/2}f(\bd_+e),         
             &\text{if $\bd_+ e \in V(\VG)$, $\bd_-e \in V_0$,}\\
        - \e^{-\im \alpha_e/2}f(\bd_-e),       
             &\text{if $\bd_- e \in V(\VG)$, $\bd_+e \in V_0$,}
      \end{cases}
    \end{equation*}
    see~\eqref{eq:twist.der}.  This is the same as extending $f \in
    \ell_2(V(\VG),\Vm)$ by $0$.  In particular, the notation $d_{\VA}$
    is justified.  Actually, the vector potential on connecting edges $e \in
    B(\VG,\G)$ can be gauged away.
   
  \item
    \label{VV.b}
    The nature of virtualising vertices is different from the process
    of virtualising edges, but dual in the sense that for virtualising
    edges, we use $d_{\alpha^-}=\iota^* \circ d_\alpha$, and for
    virtualising vertices, we use $d_{\alpha^+}=d_\alpha \circ \iota$
    with $\iota$ being the natural embedding on the space of edges and
    vertices, respectively.  As a consequence, $\Delta_{\VA}^{\VM} =
    d_{\alpha^+}^* d_{\VA} =\iota^*\Delta_\alpha^\W \iota$, i.e.,
    $\Delta_{\VA}^{\VM}$ is a compression of $\Delta_\alpha^\W$.  If
    we number the vertices $V(\G)$ such that the vertices of $V_0$
    appear at the end, then a matrix representation of
    $\Delta_\alpha^\W$ (cf.\ Subsection~\ref{MatrixDML}) has the block
    structure
    \begin{equation*}
      \Delta_\alpha^\W 
      =\begin{bmatrix} \Delta_{\VA}^{\VM}& C_{12}\\ C_{21} & C_{22} \end{bmatrix},
    \end{equation*}
    i.e., $\Delta{_{\VA}^{\VM}}$ corresponds to a principal sub-matrix
    of $\Delta_\alpha^\W$.  In particular, we have the following
    inequality for traces:
    \begin{equation*}
      \Tr\left[ \Delta{_{\VA}^{\VM}} \right] 
      \leq \Tr \left[\Delta_\alpha^\W \right].
    \end{equation*}
  \end{enumerate}
\end{remark}

We show next that the process of vertex virtualisation makes a DML
spectrally larger:
\begin{proposition}
  \label{DMDL}
  Let $\W=(\G,m)$ be a weighted finite graph with $\alpha$ a vector
  potential and $V_0\subset V$. Denote by $\VM=(\VG,\Vm)$ the
  vertex-virtualised graph with $\VG=\G - V_0$, then
  \begin{equation*}
    \Delta_\alpha^\W \preccurlyeq \Delta_{\VA}^{\VM}.
  \end{equation*}
\end{proposition}
\begin{proof}
  Let $n=|V|$ and $r=|V_0|$.  Since $\Delta_{\VA}^{\VM}
  =\iota^*\Delta_\alpha^\W \iota$ is a compression of
  $\Delta_\alpha^\W$ we can apply Cauchy's Interlacing Theorem (see
  e.g.~\cite[Corollary~II.1.5]{bhatia1987perturbation})
  and obtain
  \begin{equation*}
    \lambda_k(\Delta_\alpha^\W)
    \leq \lambda_k(\Delta_{\VA}^{\VM}) 
    \leq \lambda_{k+r}(\Delta_\alpha^\W),
    \qquad k=1,2, \dots,n-r.
  \end{equation*} 
  Since $\Delta_{\VA}^{\VM}$ acts on an $(n-r)$-dimensional Hilbert
  space we have shown that $\Delta_\alpha^\W\preccurlyeq
  \Delta_{\VA}^{\VM}$ (cf., Definition~\ref{def:PO}) using only the
  first inequality.
\end{proof}

Using the notation of the preceding proposition, we also conclude:
\begin{corollary}
  If $|V_0|=1$ then we have a complete interlacing of eigenvalues,
  i.e.,
  \begin{equation*}
    \lambda_1(\Delta_\alpha^\W)
    \leq \lambda_1(\Delta_{\VA}^{\VM})
    \leq \lambda_2(\Delta_\alpha^{\W})
    \leq \dots 
    \leq \lambda_{n-1}(\Delta_\alpha^{\W})
    \leq \lambda_{n-1}(\Delta_{\VA}^{\VM})
    \leq\lambda_n(\Delta_\alpha^\W).
  \end{equation*}
\end{corollary}

Summarising, given a weighted graph $\W=(\G,m)$ with vector potential
$\alpha$ the process of edge and vertex virtualisation produces two
different graphs $\EM=(\EG,\Em)$ respectively $\VM=(\VG,\Vm)$ with
induced potentials and such that the corresponding DMLs are spectrally
smaller respectively larger than the original one, i.e.,
\begin{equation*}
  \Delta^{\EM}_{\EA}  
  \preccurlyeq \Delta_\alpha^\W
  \preccurlyeq \Delta_{\VA}^{\VM}.
\end{equation*}


This will be the basis for the bracketing technique used later on.
Let us now specify the virtualised edge and vertex set such that the
vector potential on $\EG$ and $\VG$ becomes cohomologous to 0:

\begin{definition}
  \label{def:admissible}
  Let $\G$ be a graph and $E_0 \subset E(\G)$.  We say that a vertex
  subset $V_0 \subset V(\G)$ is \emph{in the neighbourhood of $E_0$}
  if $E_0 \subset \bigcup_{v \in V_0} E_v$, i.e., if $\bd_+ e \in V_0$
  or $\bd_- e\in V_0$ for all $e\in E_0$.
\end{definition}
Later on $E_0$ will be the set of connecting edges of a periodic graph, and we
will choose $V_0$ to be as small as possible to guarantee the
existence of spectral gaps (this set is in general not unique).

\begin{theorem}
  \label{thm:technique}
  Let $\W=(\G,m)$ be a  finite weighted graph and $E_0\subset E(\G)$.  Then
  for any vector potential $\alpha$ supported in $E_0$ and any set
  $V_0$ in the neighbourhood of $E_0$ of vertices we have
  \begin{equation}
    \Delta^{\EM}
    \preccurlyeq \Delta^{\W}_\alpha
    \preccurlyeq \Delta^{\VM},
  \end{equation}
  where $\EM=(\EG,\Em)$ with $\EG=\G-E_0$ and $\VM=(\VG,\Vm)$ with
  $\VG=\G - V_0$.  In particular, we have the spectral localising
  inclusion
  \begin{equation}
    \label{eq:loc.spec}
    \sigma(\Delta^{\W}_\alpha) 
    \subset  J
    := J(\Delta^{\EM},\Delta^{\VM})
    = \bigcup_{k=1}^{|V(\G)|} 
        \bigl[\lambda_k(\Delta^{\EM}),\lambda_k(\Delta^{\VM})\bigr],
  \end{equation}
  where $J$ does not depend on the vector potential.
\end{theorem}
\begin{proof}
  Let $\W=(\G,m)$ be a weighted graph and $E_0\subset E(\G)$.  For any
  vector potential $\alpha$, we have by Proposition~\ref{VirtualEdges}
  that $\Delta^{\EM}_{\EA} \preccurlyeq \Delta^{\W}_\alpha$, where
  $\EM=(\EG,\Em)$. If, in addition, $\alpha$ is supported in $E_0$,
  then $\alpha_e=0$ for any $E(\EG)=E(\G)\setminus E_0$ hence,
 \begin{equation*}
    \Delta^{\EM}
    \preccurlyeq \Delta^{\W}_\alpha.
  \end{equation*}
  For $V_0$ in the neighbourhood of $E_0$ and any vector potential
  $\alpha$, we have by Proposition~\ref{DMDL} that $
  \Delta^{\W}_\alpha \preccurlyeq \Delta^{\VM}_{\VA}$, where
  $\VM=(\VG,\Vm)$. If $\alpha$ is supported in $E_0$ and since $E_0
  \subset \bigcup_{v \in V_0} E_v$, i.e., if $\bd_+ e \in V_0$ or
  $\bd_- e\in V_0$ for all $e\in E_0$, then the vector potential
  ${\VA}$ can be gauged away, hence
 \begin{equation*}
\Delta^{\W}_\alpha
    \preccurlyeq \Delta^{\VM}.
  \end{equation*}   
  By construction we have that the operators $\Delta^{\W^\pm}$
  specifying the boundary of the bracketing intervals are independent
  of the vector potentials.  Finally, the bracketing inclusion follows
  from the Definition~\ref{def:brack.int}.
\end{proof}

\begin{remark}\label{RemarkTree}
  If $\EG=\G\setminus E_0$ is a tree, then we can allow vector
  potentials supported on all edges in $E$, since
  $\Delta_\alpha^{\EM}$ is unitarily equivalent to $\Delta^{\EM}$.
  Similarly, on $\VG$, the vector potential is cohomologous to $0$, as
  the remaining loops are also removed by virtualising the vertices in
  $V_0$ (recall that $V_0$ is in the neighbourhood of $E_0$, see
  Definition~\ref{def:admissible}).  In particular, we have
  \begin{equation}
    \label{eq:bracketing}
    \UVP\sigma(\Delta^{\W}_\alpha) 
    \subset  \bigcup_{k=1}^{|V(\G)|}
    [\lambda_k(\Delta^{\EM}),\lambda_k(\Delta^{\VM})]
    =: J
  \end{equation}
  for any vector potential $\alpha$ on $\G$.  Taking complements
  gives
  \begin{equation}
    \label{eq:j.gaps}
    [0,2\rho_\infty] \setminus J 
    \subset [0,2\rho_\infty] \setminus \UVP \sigma(\Delta_\alpha^\W).
  \end{equation}
\end{remark}

%
\section{Magnetic spectral gaps}
\label{sec:MSG}
%

We will apply in this section
the spectral ordering method mentioned in the preceding section to
localise the spectrum of the DML on certain bracketing intervals. With
this technique we will be able to prove the existence of spectral gaps
for certain periodic Laplacians. We will also consider in this section only
finite graphs.

We begin by making precise several notions of spectral gaps.
Denote by $\sigma(T)$ and $\rho(T)=\C \setminus \sigma(T)$ the spectra
and the resolvent set of a self-adjoint operator $T$, respectively.
Recall that $\sigma(\Delta^\G_\alpha) \subset [0,2\rho_\infty]$,
where $\rho_\infty$ denotes the supremum of the relative weight,
see Equation~\eqref{eq:rel.weight}.
Remark~\ref{RemarkTree} suggests the following natural question: when do we have $J \subsetneq
[0,2\rho_\infty]$? This question motivates the following definition.

\begin{definition}
  \label{def:spec.gaps}
  Let $\W=(\G,m)$ be a weighted graph.
  \begin{enumerate}
  \item The \emph{spectral gaps set} of $\W$ is defined by
    \begin{equation*}
      \SG^\W
      =[0,2\rho_\infty] \setminus \sigma(\Delta^\W)
      =[0,2\rho_\infty]\cap  \rho(\Delta^\W).
    \end{equation*}
    
  \item The \emph{magnetic spectral gaps set} of $\W$ is defined by
    \begin{equation*}
      \MSG^\W
      =[0,2\rho_\infty] \setminus \UVP \sigma(\Delta^\W_\alpha)
      =\bigcap_{\alpha\in\mathcal{A}(\G)} \rho(\Delta^\W_\alpha) \cap [0,2\rho_\infty].
    \end{equation*}
  \end{enumerate}
  where the union is taken over all the vector potential $\alpha$
  acting on $\G$.
\end{definition}


We have the following elementary properties:
\begin{itemize}
\item $\MSG^\W\subset \SG^\W$.  In particular, if $\SG^\W=\emptyset$,
  then $\MSG^\W=\emptyset $.  Moreover, if $\MSG^\W \neq \emptyset$,
  then $\SG^\W \neq \emptyset$.

\item If $\G$ is a tree, then $\MSG^\W = \SG^\W$, as all DMLs are
  unitarily equivalent with $\Delta^\W$ (the usual Laplacian).
\end{itemize}

\begin{example}
  If $\W=(\G,m)$ is a weighted graph where $\G$ is either the
  $\Z^n$-lattice or the graphene lattice (hexagonal lattice consisting
  of carbon atoms, see Figure~\ref{subfig:Z2}
  and~\ref{subfig:graphene}) both with standard weights, then
  $\sigma(\Delta^{\G})=[0,2]$. 
  Hence, the set of spectral gaps is empty, i.e., $\SG^\G=\emptyset$
  and hence $\MSG^\G=\emptyset$.

\begin{figure}[h]
    \label{fig:graphene}
\centering
\subcaptionbox{The $\Z^2$-lattice. \label{subfig:Z2}}{
\begin{tikzpicture}[auto, vertex/.style={circle,draw=black!100,fill=black!100, thick,
                inner sep=0pt,minimum size=1mm},scale=1.25]	                
    \node (D) at (-.75,0) [vertex,inner sep=.25pt,minimum size=.25pt,label=above:]{};
    \node (D) at (-.5,0) [vertex,inner sep=.25pt,minimum size=.25pt,label=above:]{};
    \node (D) at (-.25,0) [vertex,inner sep=.25pt,minimum size=.25pt,label=above:]{};
    \node (A) at (0,0) [vertex,label=below:] {};
    \node (B) at (.5,0) [vertex,label=below:]{};
    \node (C) at (1,0) [vertex,label=below:]{};
    \node (D) at (1.5,0) [vertex,label=below:]{};
    \draw[-latex] (A) to[] node[above] {} (B);
	\draw[-latex] (B) to[]  node[above] {} (C);
	\draw[-latex] (C) to[]  node[above] {} (D);
	\node (E) at (1.75,0) [vertex,inner sep=.25pt,minimum size=.25pt,label=above:]{};
	\node (F) at (2,0) [vertex,inner sep=.25pt,minimum size=.25pt,label=above:]{};
	\node (G) at (2.25,0) [vertex,inner sep=.25pt,minimum size=.25pt,label=above:]{};
    \node (D1) at (-.75,.5) [vertex,inner sep=.25pt,minimum size=.25pt,label=above:]{};
    \node (D1) at (-.5,.5) [vertex,inner sep=.25pt,minimum size=.25pt,label=above:]{};
    \node (D1) at (-.25,.5) [vertex,inner sep=.25pt,minimum size=.25pt,label=above:]{};           
    \node (A1) at (0,.5) [vertex,label=below:] {};
    \node (B1) at (.5,.5) [vertex,label=below:]{};
    \node (C1) at (1,.5) [vertex,label=below:]{};
    \node (D1) at (1.5,.5) [vertex,label=below:]{};
    \draw[-latex] (A1) to[] node[above] {} (B1);
	\draw[-latex] (B1) to[]  node[above] {} (C1);
	\draw[-latex] (C1) to[]  node[above] {} (D1);
	\node (E1) at (1.75,.5) [vertex,inner sep=.25pt,minimum size=.25pt,label=above:]{};
	\node (F1) at (2,.5) [vertex,inner sep=.25pt,minimum size=.25pt,label=above:]{};
	\node (G1) at (2.25,.5) [vertex,inner sep=.25pt,minimum size=.25pt,label=above:]{};
    \draw[-latex] (A) to[] node[left] {} (A1);
	\draw[-latex] (B) to[]  node[left] {} (B1);
	\draw[-latex] (C) to[]  node[left] {} (C1);
	\draw[-latex] (D) to[]  node[left] {} (D1);
    \node (D2) at (-.75,-.5) [vertex,inner sep=.25pt,minimum size=.25pt,label=above:]{};
    \node (D2) at (-.5,-.5) [vertex,inner sep=.25pt,minimum size=.25pt,label=above:]{};
    \node (D2) at (-.25,-.5) [vertex,inner sep=.25pt,minimum size=.25pt,label=above:]{};             
    \node (A2) at (0,-.5) [vertex,label=below:] {};
    \node (B2) at (.5,-.5) [vertex,label=below:]{};
    \node (C2) at (1,-.5) [vertex,label=below:]{};
    \node (D2) at (1.5,-.5) [vertex,label=below:]{};
    \draw[-latex] (A2) to[] node[above] {} (B2);
	\draw[-latex] (B2) to[]  node[above] {} (C2);
	\draw[-latex] (C2) to[]  node[above] {} (D2);
	\node (F2) at (1.75,-.5) [vertex,inner sep=.25pt,minimum size=.25pt,label=above:]{};
	\node (E2) at (2,-.5) [vertex,inner sep=.25pt,minimum size=.25pt,label=above:]{};
	\node (G2) at (2.25,-.5) [vertex,inner sep=.25pt,minimum size=.25pt,label=above:]{};
    \draw[-latex] (A2) to[] node[left] {} (A);
	\draw[-latex] (B2) to[]  node[left] {} (B);
	\draw[-latex] (C2) to[]  node[left] {} (C);
	\draw[-latex] (D2) to[]  node[left] {} (D);
    \node (Au) at (0,-.75) [vertex,inner sep=.25pt,minimum size=.25pt,label=above:]{};
    \node (Au) at (0,-1) [vertex,inner sep=.25pt,minimum size=.25pt,label=above:]{};
    \node (Au) at (0,-1.25) [vertex,inner sep=.25pt,minimum size=.25pt,label=above:]{};
    \node (Au) at (1,-.75) [vertex,inner sep=.25pt,minimum size=.25pt,label=above:]{};
    \node (Au) at (1,-1) [vertex,inner sep=.25pt,minimum size=.25pt,label=above:]{};
    \node (Au) at (1,-1.25) [vertex,inner sep=.25pt,minimum size=.25pt,label=above:]{};
    \node (Au) at (0.5,-.75) [vertex,inner sep=.25pt,minimum size=.25pt,label=above:]{};
    \node (Au) at (0.5,-1) [vertex,inner sep=.25pt,minimum size=.25pt,label=above:]{};
    \node (Au) at (0.5,-1.25) [vertex,inner sep=.25pt,minimum size=.25pt,label=above:]{};
    \node (Au) at (1.5,-.75) [vertex,inner sep=.25pt,minimum size=.25pt,label=above:]{};
    \node (Au) at (1.5,-1) [vertex,inner sep=.25pt,minimum size=.25pt,label=above:]{};
    \node (Au) at (1.5,-1.25) [vertex,inner sep=.25pt,minimum size=.25pt,label=above:]{};
    \node (Au) at (0,.75) [vertex,inner sep=.25pt,minimum size=.25pt,label=above:]{};
    \node (Au) at (0,1) [vertex,inner sep=.25pt,minimum size=.25pt,label=above:]{};
    \node (Au) at (0,1.25) [vertex,inner sep=.25pt,minimum size=.25pt,label=above:]{};
    \node (Au) at (1,.75) [vertex,inner sep=.25pt,minimum size=.25pt,label=above:]{};
    \node (Au) at (1,1) [vertex,inner sep=.25pt,minimum size=.25pt,label=above:]{};
    \node (Au) at (1,1.25) [vertex,inner sep=.25pt,minimum size=.25pt,label=above:]{};
    \node (Au) at (0.5,.75) [vertex,inner sep=.25pt,minimum size=.25pt,label=above:]{};
    \node (Au) at (0.5,1) [vertex,inner sep=.25pt,minimum size=.25pt,label=above:]{};
    \node (Au) at (0.5,1.25) [vertex,inner sep=.25pt,minimum size=.25pt,label=above:]{};
    \node (Au) at (1.5,.75) [vertex,inner sep=.25pt,minimum size=.25pt,label=above:]{};
    \node (Au) at (1.5,1) [vertex,inner sep=.25pt,minimum size=.25pt,label=above:]{};
    \node (Au) at (1.5,1.25) [vertex,inner sep=.25pt,minimum size=.25pt,label=above:]{};           
  \end{tikzpicture}} \hspace{.5in}\subcaptionbox{The graphene
  graph. \label{subfig:graphene}}{
  \includegraphics[width=0.35\linewidth]{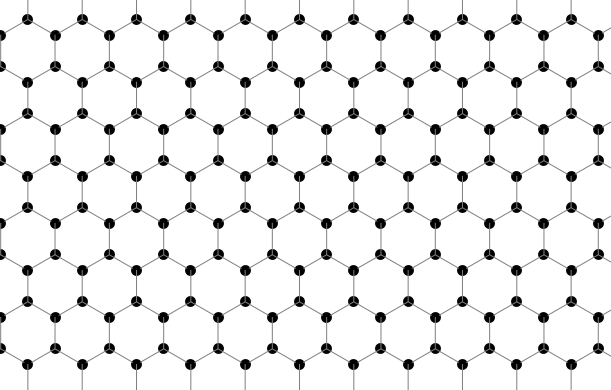}}\\
\subcaptionbox{The graphane graph. \label{subfig:graphane}}{
        \includegraphics[width=0.35\linewidth]{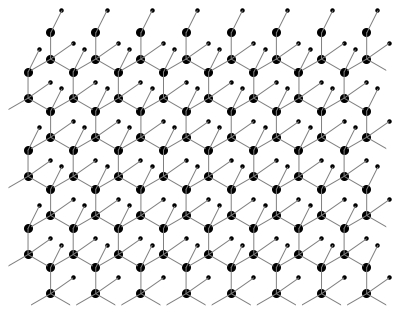}}
    \caption{Examples of $\Z_2$ periodic graphs.} 
  \end{figure}
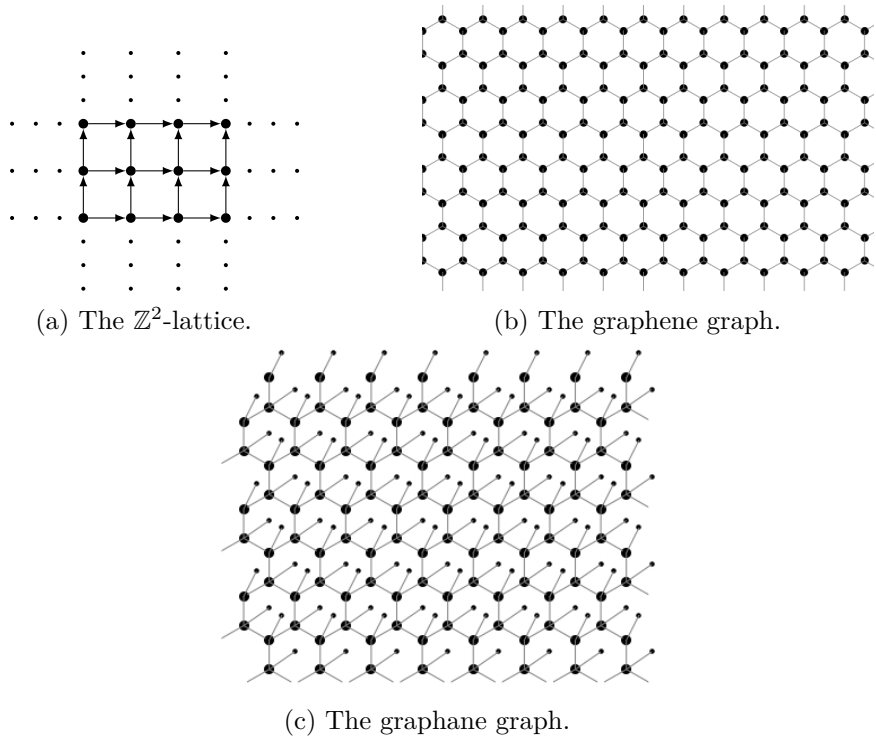
\end{example}

The graph\emph{a}ne is just the decoration of the graph\emph{e}ne
adding an hydrogen atom for each carbon atom (see
Figure~\ref{subfig:graphane}).
Here, we have $\sigma(\Delta^\G)=[0,3/4] \cup [5/4,2]$ for the
standard weight, hence $\SG^\G=(3/4,5/4)$.

\begin{definition}
  We say that a graph $\G$ has a \emph{centre vertex} if there exists
  a vertex $v_0$ and a subset $A(v_0) \subset E_{v_0}$ such that $\G -
  A(v_0)$ is a tree (and in particular connected).  We call the edges
  in $A(v_0)$ \emph{cycle edges}.
\end{definition}
A centre vertex is a vertex where all cycles of the graph $\G$ meet.
Note that $\{v_0\}$ is in the neighbourhood of $A(v_0)$ (see
Definition~\ref{def:admissible}).

Is clear that if $\G$ is a tree, then $\G$ does not have a centre
vertex (as a tree is connected).  Moreover, if $\G$ is a cycle, then
any vertex is a centre vertex. 
By definition, if $\G$ has centre vertex $v_0$, then $\Betti(\G)=|A(v_0)|$.

We are now able to prove the following sufficient condition for the
existence of magnetic gaps (i.e., for $\MSG\ne \emptyset$). We will 
use this result (in the case of standard weights) in
the examples presented in Section~\ref{sec:examples}.
Recall that we allow loops and multiples edges in the graph $\G$.

\begin{theorem}
  \label{TheoremSpectralGaps2}
  Let $\W=(\G,m)$ be a weighted graph.  If $v_0$ is a centre vertex
  with cycle edges $A(v_0)$ and let
  \begin{equation}
    \label{eq:weight.cond}
    \delta
    := \rho(v_0) 
     - \sum_{e\in A(v_0)}\frac{m_e}{m((v_0)_e)} -\frac{m( A(v_0))}{m(v_0)}
  \end{equation}
  where $\rho(v_0)=m(E_{v_0}))/m(v)$ is the relative weight at $v_0$.
  Then the Lebesgue measure of the magnetic spectral gaps set is larger
  or equal to $\delta$. In particular, if $\delta>0$, then $\MSG^\W\neq\emptyset$.
\end{theorem}
\begin{proof}
  Let $\alpha$ be any vector potential and consider the virtualised graphs
  $\EG=\G -A(v_0)$ and $\VG=\G-\{v_0\}$.  As $\EG$ is a
  tree, we have that $[0,2\rho_\infty] \setminus J\subset \MSG^\W$,
  where $J$ is the union of the bracketing interval (cf., Eq.~\eqref{eq:j.gaps}).  
  In particular, the
  measure of $[0,2\rho_\infty] \setminus J$ is smaller or equal to the
  measure of $\MSG^\W$.  As $\lambda_1(\Delta^{\EM})=0$, the measure
  of $[0,2\rho_\infty] \setminus J$ can be estimated from below
  by 
  \begin{align}
    \nonumber
    \sum_{k=1}^{n-1}
      \bigl(\lambda_{k+1}\bigl(\Delta^{\EM}\bigr)
      - \lambda_k\bigl(\Delta^{\VM}\bigr)\bigr)
    =& \sum_{k=1}^{n}\lambda_{k}\bigl(\Delta^{\EM}\bigr)
     -\sum_{k=1}^{n-1}\lambda_k\bigl(\Delta^{\VM}\bigr)\\
    \label{eq:step0}
    &= \Tr\bigl(\Delta^{\EM} \bigr) -   \Tr\bigl(\Delta^{\VM} \bigr),          
  \end{align} 
  so we need to calculate $\Tr\bigl(\Delta^{\EM} \bigr)$ and
  $\Tr(\Delta^{\VM})$ (see Proposition~\ref{prp:key-obs}).

  \myparagraph{Step 1: Trace of $\Delta^{\EM}$.}  Let
  $\EM=(\EG,\Em)$ and recall that
  $V(\G-A(v_0))=V(\G)$, $E(\G-A(v_0))=E(\G)\setminus A(v_0) $; the
  weights on $V(\G-A(v_0))$ and $E(\G-A(v_0))$ coincide with the
  corresponding weights on $\W$.  The relative weights of $\EM$ are
  \begin{equation*}
    \rho^-(v)
    =
    \begin{cases}
      \rho^{\W}(v)-\dfrac{m \left( E(v)\right) +
     m \left( A(v_0)\right) }{m(v)}, &\text{if $v=v_0$,}\\[2ex]
      \rho^{\W}(v)-\dfrac{m\left(  A(v_0)\cap E_v\right) }{m(v)}, 
      &\text{if $v\in B_{v_0}$,}\\[2ex]
      \rho^{\W}(v), &  \text{otherwise,}
    \end{cases}
  \end{equation*}
  where 
  \begin{equation*}
    B_{v_0}
    = \{v \in V(\G) \mid 
        v=(v_0)_e \text{ for some } e\in A(v_0) \text{ with }
             v\neq v_0 \}.
  \end{equation*}
  Since $v_0$ is a centre vertex, the only loops (that $\G$ could
  possibly have) must be attached to $v_0$.  The trace of
  $\Delta^{\EM}$ is now
  \begin{align}
    \nonumber
    \Tr\bigl(\Delta^{\EM}\bigr)&=\sum_{k=1}^n \lambda_k\bigl(\Delta^{\EM} \bigr)
    =\sum_{v\in V(\G)} \rho^-(v)\\
    \label{eq:step1}
    &= \sum_{v\in V(\G)}\rho^{\W}(v)-\dfrac{m \left( E(v_0)\right) +
         m \left( A(v_0)\right) }{m(v_0)}-
      \sum\limits_{v\in B_{v_0}}\dfrac{m\left(  A(v_0)\cap E_v\right) }{m(v)}.
  \end{align}

  \myparagraph{Step 2: Trace of $\Delta^{\VM}$.} Let $\W^+=\left(\VG,m^+ \right) $, then the trace of
  $\Delta^{\VM}$ is given by
  \begin{equation}
    \label{eq:step2}
    \Tr\bigl(\Delta^{\VM}\bigr)
    =\sum_{k=1}^{n-1} \lambda_k\bigl(\Delta^{\W}\bigr)
    =\sum_{\substack{v\in V(\G) \\v\neq v_0}} \rho^{\W}(v).
  \end{equation}
  
  Combining Equations~\eqref{eq:step0}, \eqref{eq:step1}
  and~\eqref{eq:step2} we obtain
  \begin{align*}
    \Tr\bigl(\Delta^{\EM} \bigr) - \Tr\bigl(\Delta^{\VM} \bigr)
    & =\rho^{\W}(v_0)
             -\frac{m(E(v_0)) + m(A(v_0))}{m(v_0)}
             -\sum_{v\in B_{v_0}} \frac {m(A(v_0)\cap E_v)}{m(v)}\\
    & =\rho^{\W}(v_0)
             -\frac{m(E(v_0)) + m \left( A(v_0)\right)}{m(v_0)}
             - \sum_{e\in A(v_0)\setminus E(v_0)}\frac{m_e}{m((v_0)_e)}\\         
    & =\rho^{\W}(v_0)
             -\frac{m \left( A(v_0)\right) }{m(v_0)}
             - \sum_{e\in A(v_0)}\dfrac{m_e}{m((v_0)_e)}=\delta         
  \end{align*}
  as defined in Equation~\eqref{eq:weight.cond}.  The last assertion
  is a simple consequence.
\end{proof}
\vspace{1cm}

\begin{remark}
  \label{rem:gen.weights}
  \indent
  \begin{enumerate}
  \item In the proof, we used the spectral localising
    inclusion~\eqref{eq:bracketing}.  If the weighted graph is
    bipartite and if the weight is normalised (or more generally, if
    the relative weight is constant), and if we find $B\subset
    \MSG^\W$ then we have also $\kappa(B)\subset \MSG^\G $ by
    Proposition~\ref{prp:bipartite.sym}.
  \item For applications, in particular, for the examples of Section~\ref{sec:examples},
    we explicitly write Condition~\ref{eq:weight.cond} for the most
    important weights (see Section~\ref{ssec:weight}). Let $v_0$ be a centre vertex
    with cycle edges $A(v_0)$:
    \begin{enumerate}[label=(\roman*)]
    \item If the graph has the \emph{standard weights}, the condition
      becomes:
      \begin{equation}
        \delta= 1 - \sum_{e\in A(v_0)}\dfrac{1}{\deg((v_0)_e)}-
        \dfrac{|A(v_0)|}{\deg(v_0)}\;,
      \end{equation}
      where $|A(v_0)|$ denote the cardinality of the set $A(v_0)$.
    \item Now, if we have the \emph{combinatorial weights}, the
      condition becomes simply:
      \begin{equation}
       \delta= \deg(v_0) - 2\;|A(v_0)|\;.
      \end{equation}  
    \item For the \emph{electric circuit weights}, the condition is:
      \begin{equation}
        \delta = m(E_{v_0}) - 2\;m(A(v_0))\;.
      \end{equation}    
    \item For the \emph{normalised weights}, the condition is:
      \begin{equation}
       \delta= 1 - \sum_{e\in A(v_0)}\dfrac{m_e}{m(E_{(v_0)_e})}
	         -\dfrac{m( A(v_0))}{m(E_{v_0})}\;.
      \end{equation}      
    \end{enumerate}
    In all of the previous cases, if $\W$ is a graph with the
    corresponding weights and meets the condition $\delta>0$, then we can assure
    the existence of magnetic spectral gaps, i.e., $\MSG^\W\ne
    \emptyset$.
  \end{enumerate}
\end{remark}

\begin{example}\label{exa:cycle-edge}
  Consider the next two graphs in Figure~\ref{fig:theorem}, both with
  the standard weights.  In both graphs, $v_0$ is a centre vertex with
  cycle edges $A(v_0)=\left\lbrace e_1,e_2\right\rbrace $. The
  strategy to produce gaps is to raise the degree of the vertices
  $(v_0)_{e_1}$ and $(v_0)_{e_2}$.  In the first case of the
  Figure~\ref{subfig:nogaps} we have no magnetic spectral gaps, i.e.,
  $\MSG^{\G_1}=\emptyset$ while for the second graph~\ref{subfig:gaps}
  we have 
  \begin{equation*}
   \delta=1- \dfrac{1}{\deg((v_0)_{e_1})}- \dfrac{1}{\deg((v_0)_{e_2})}-
    \dfrac{2}{\deg(v_0)}=1-\frac 14-\frac 15-\frac 24=\frac {1}{20}
    >0,
  \end{equation*}
  then $\MSG^{\G_2}\neq \emptyset$ as a consequence of
  Theorem~\ref{TheoremSpectralGaps2}.  This example also shows that
  Condition~\eqref{eq:weight.cond} is sufficient but not necessary:
  consider the graph $\G_2$ with only one decorating edge at each
  vertex $(v_0)_{e_1}$ and $(v_0)_{e_2}$.  The corresponding graph
  still has a spectral gap, although $\delta=1-1/3-1/3-2/4=-1/6<0$.

  \begin{figure}[h]
      \centering \subcaptionbox{The graph
        $\G_1$. \label{subfig:nogaps}}%
      [.3\linewidth]{\begin{tikzpicture}[auto,
          vertex/.style={circle,draw=black!100,fill=black!100, thick,
            inner sep=0pt,minimum size=1mm},scale=.5] \node (0) at
          (2,2) [vertex,label=above:] {}; \node (1) at (0,0)
          [vertex,label=left:] {}; \node (2) at (4,0)
          [vertex,black,label=above:{\color{black}{$v_0$}}] {}; \node (3)
          at (2,-2) [vertex,label=above:] {}; \node (4) at (8,0)
          [vertex,label=below:] {}; \node (5) at (6,-2)
          [vertex,label=below:] {}; \path [-latex](1) edge node[below]
          {} (0); \path [-latex,black](0) edge node[below] {$e_1$} (2);
          \path [-latex](2) edge node[below] {} (3); \path [-latex](3)
          edge node[below] {} (1); \path [-latex](2) edge node[below]
          {} (4); \path [-latex](4) edge node[below] {} (5); \path
          [-latex](2) edge node[below] {} (4); \path [-latex,black](5)
          edge node[below] {$e_2$} (2);
         \end{tikzpicture}}
    \subcaptionbox{The graph $\G_2$.\label{subfig:gaps}}
      [.3\linewidth]{\begin{tikzpicture}[auto, vertex/.style={circle,draw=black!100,fill=black!100, thick,
                            inner sep=0pt,minimum size=1mm},scale=.5]
                \node (0) at (2,2) [vertex,label=above:] {};
             	  \node (01) at (.5,2) [vertex,label=above:] {};
                \node (02) at (3.5,2) [vertex,label=above:] {};          
                \node (1) at (0,0) [vertex,label=left:] {};
                \node (2) at (4,0) [vertex,black,label=above:{\color{black}{$v_0$}}] {};
                \node (3) at (2,-2) [vertex,label=above:] {};
                \node (4) at (8,0) [vertex,label=below:] {};
                \node (5) at (6,-2) [vertex,label=below:] {};
            	\path [-latex](1) edge node[below] {} (0);
            	\path [-latex](0) edge node[below] {} (01);
            	\path [-latex](0) edge node[below] {} (02);
            	\path [-latex,black](0) edge node[below] {$e_1$} (2);             
            	\path [-latex](2) edge node[below] {} (3);
            	\path [-latex](3) edge node[below] {} (1);
            	\path [-latex](2) edge node[below] {} (4);
            	\path [-latex](4) edge node[below] {} (5);
            	\path [-latex](2) edge node[below] {} (4);   	
            	\path [-latex,black](5) edge node[below] {$e_2$} (2); 
             	  \node (01) at (7.5,-2) [vertex,label=above:] {};
                \node (02) at (4.5,-2) [vertex,label=above:] {};   
                \node (03) at (6,-.5) [vertex,label=above:] {};
               \path [-latex](5) edge node[below] {} (01);
            	\path [-latex](5) edge node[below] {} (02);
            	\path [-latex](5) edge node[below] {} (03);                 	  	
            \end{tikzpicture}  } 
          \caption{Producing magnetic spectral gaps by decoration.
            The graph $\G_2$ is obtained from $\G_1$ by adding pendant
            edges at $(v_0)_{e_1}$ and
            $(v_0)_{e_2}$. \label{fig:theorem}}
          \label{fig:graphgaps}
  \end{figure}
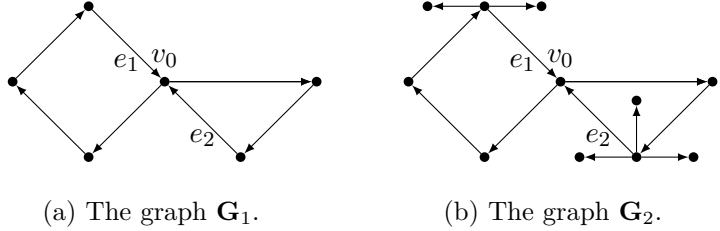
\end{example}
As a consequence, we have the following topological characterisation
for the existence of magnetic spectral gaps for graphs with Betti
number $1$: %

\begin{corollary}
  \label{cor:spectral_gap}
  Let $\W=(\G,m)$ be a weighted graph with standard weights and Betti
  number $\Betti(\G)=1$.  Then the following conditions are equivalent:
  \begin{enumerate}
  \item 
    \label{spectral_gap.a}
    $\G$ has magnetic spectral gap (i.e., $\MSG^\G\neq \emptyset$);
  \item 
    \label{spectral_gap.b}
    $\G$ is not a cycle graph;
  \item 
    \label{spectral_gap.c}
    $\G$ has a vertex of degree $1$.
  \end{enumerate}
\end{corollary}
\begin{proof}
  ``\eqref{spectral_gap.a}$\Rightarrow$\eqref{spectral_gap.b}'':
  Suppose that $\G=C_n$.  Let now $\lambda
  \in [0,2]$ and $t\in [0,2\pi]$ be such that $\cos t=1-\lambda$.
  Consider a vector potential $\alpha$ given by $\alpha_e=t$ for all
  $e \in E(C_n)$.  We will show that
  $\lambda\in\sigma(\Delta_\alpha^\G)$. In fact, consider $\1(v)=1$
  for all $v\in V(\G)$, then
  \begin{equation*}
    (\Delta_\alpha^\W \1)(v)
    =1-\frac{\e^{-it}+\e^{it}}2
    =1-\cos t
    =\lambda \cdot \1(v). 
  \end{equation*}
  We have shown that $[0,2] \subset \UVP
  \sigma(\Delta^\G_\alpha)$, i.e., $\MSG^\G = \emptyset$.
  
  ``\eqref{spectral_gap.b}$\Rightarrow$\eqref{spectral_gap.c}'': Using
  the fact that $\Betti(\G)=1$, one can prove this by induction on the
  number of vertices.
  
  ``\eqref{spectral_gap.c}$\Rightarrow$\eqref{spectral_gap.a}'': Since
  $\G$ has Betti number $\Betti(\G)=1$ and since $\G$ has a vertex of
  degree $1$, there exists $v_0 \in V(\G)$ such that $v_0$ belongs to
  the cycle with $\deg v_0 \ge 3$, and it is adjacent with $v_1 \in
  V(\G)$ by an edge $e_1 \in E(\G)$ with $\deg v_1 \geq 2$.  Moreover,
  $v_0$ is a centre vertex with cycle edge $A(v_0)=\left\lbrace e_1
  \right\rbrace$.  As
  \begin{equation*}
    1-\sum_{e\in A(v_0)}\dfrac{1}{\deg((v_0)_e)}-
         \dfrac{|A(v_0)|}{\deg(v_0)}=1- \frac1{\deg_\G v_1}- \frac1{\deg_\G v_0} 
    \geq1- \frac12-\frac13=\dfrac16>0,
  \end{equation*} 
  we conclude from Theorem~\ref{TheoremSpectralGaps2} that $\MSG^\G\ne
  \emptyset$.
  
\end{proof}

\begin{remark}
  \label{rem:Gaps}
  \indent
  \begin{enumerate}
  \item Corollary~\ref{cor:spectral_gap} holds also for
    combinatorial weights: for
    ``\eqref{spectral_gap.a}$\Rightarrow$\eqref{spectral_gap.b}'' note
    that $C_n$ is a regular graph, hence the spectrum of the standard
    weight and the combinatorial weight is just related by a simple
    scaling.  For
    ``\eqref{spectral_gap.c}$\Rightarrow$\eqref{spectral_gap.a}'' note
    that the condition on the weights becomes $2=2|A(v_0)| <3\le \deg
    v_0$ (choose $v_0$ to be the vertex of degree larger than $2$). Finally
      ``\eqref{spectral_gap.b}$\Rightarrow$\eqref{spectral_gap.c}'' is 
      independent of the weights.
  \item If the graph $\G$ has the electric circuit weights, then
    ``\eqref{spectral_gap.a}$\Rightarrow$\eqref{spectral_gap.c}'' is
    no longer true.  In fact choose $\G=C_6$ as in
    Figure~\ref{fig:smaller} with $m_e=1$ if $e\neq e_1$ and
    $m_{e_1}=2$. For example, an easy calculation show that $0$ and $1\in \sigma(\Delta^\W)$ 
    but $0.3\in \MSG^\W$, then $\MSG^\W\neq \emptyset$ even if there is
    no any vertex of degree $1$.
  \item If the graph $\G$ has normalised weights, then in general
    ``\eqref{spectral_gap.a}$\Rightarrow$\eqref{spectral_gap.c}'' is
    not true. Choose $\G=C_6$ as in Figure~\ref{fig:smaller}
    with the following weights: on the edges $m_e=1$ if $e\neq e_1$ and $m_{e_1}=2$,
    on the vertices $m(v_i)=2$ for $i=1,2$ and $m(v_i)=1$ for $i=3,4,5,6$.
    It is easy to check that $0$ and $1\in \sigma(\Delta^\W)$ with
    $1/2\in \MSG^\W$. Then $\MSG^\W\neq
    \emptyset$ but again, there is no vertex of degree $1$.	
 \end{enumerate}
\end{remark}

\begin{example}
  Let $\W'=(\G',m')$ where $\G'=C_6$ and $m'$ are the standard
  weights, then by Corollary~\ref{cor:spectral_gap} we have
  $\MSG^{\G'}=\emptyset$. In order to create magnetic spectral gaps we
  add a new edge.  Let now $\W=(\G,m)$ where $\G$ is the graph $C_6$
  with an edge added to the cycle (see Figure~\ref{fig:cycle.deco})
  and $m$ the standard weights.  Then the Laplacian on $\G $ has a
  magnetic spectral gap by Corollary~\ref{cor:spectral_gap}. Now
  using the bracketing technique of Theorem~\ref{thm:technique} we can
  localise the position of the gaps.
  
  Consider $E_0=\{e_1\}$, and recall that
  any vector potential $\alpha$  can be supported on $e_1$. Consider also the
  edge virtualised weighted graph $\EM=(\EG,\Em)$ with $\EG=\G-E_0$. Then
  its spectrum is:
 \begin{equation*}
   \sigma \bigl( \Delta^{\EM}\bigr) 
   \approx \{0,0.116,0.5,0.713,1.145,1.638,1.889\}.
 \end{equation*}  
 Now, we have that $V_0=\{v_1\}$ is in the neighbourhood of $E_0$. Now
 consider the vertex virtualised weighted graph $\VM=(\VG,\Vm)$ with
 $\VG=\G - V_0$, then its spectrum is:
 \begin{equation*}
   \sigma \bigl(\Delta^{\VM}\bigr) 
   \approx \{0.121,0.358,0.744,1.256,1.642,1.879,2\}.
 \end{equation*}   
 Therefore, the bracketing intervals in which we can localise the
 spectrum is given by (see Figure~\ref{fig:cycle.deco}):
  \begin{align*}
    J_1\approx 
         \left[0,0.121\right],
         \quad 
    J_2\approx 
        \left[0.116,0.358\right],
        \quad
    J_3\approx 
         \left[0.5,0.744\right],
         \quad
    J_4\approx  
         \left[0.713,1.256\right],\\
     J_5\approx 
         \left[1.145,1.642\right],
         \quad
     J_6\approx 
          \left[1.638,1.879\right],
          \quad and \quad   
     J_7\approx  
          \left[1.889,2\right].                     
  \end{align*} 
 
  In conclusion, we have the following spectral localising  inclusion for any vector potential $\alpha$:
 
  \begin{equation}
    \UVP \sigma(\Delta^{\G}_\alpha) 
    \subset  J
    = \bigcup_{k=1}^{7} 
        J_i\subsetneq [0,2].
  \end{equation}  
  
  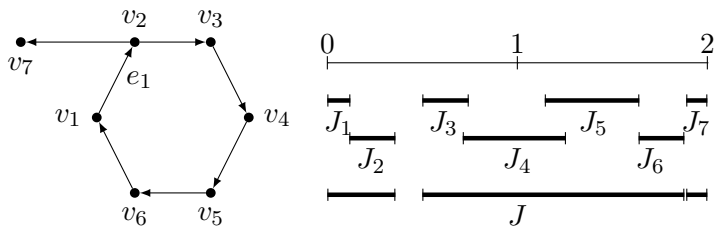
\begin{figure}[h]
      \centering
    \begin{tikzpicture}[auto, vertex/.style={circle,draw=black!100,fill=black!100, thick,
                    inner sep=0pt,minimum size=1mm}]
        \node (A) at (1/2,-2/2) [vertex,label=below:$v_5$] {};
        \node (B) at (-1/2,-2/2) [vertex,label=below:$v_6$] {};
        \node (C) at (-2/2,0) [vertex,label=left:$v_1$] {};
        \node (D) at (-1/2,2/2) [vertex,label=above:$v_2$] {};
        \node (E) at (1/2,2/2) [vertex,label=above:$v_3$] {};
        \node (F) at (2/2,0) [vertex,label=right:$v_4$] {};
        \node (G) at (-2,1) [vertex,label=below:$v_7$] {};
        
    	\path [-latex](C) edge node[right] {$e_1$} (D);
    	\path [-latex](D) edge node[below] {} (E);
    	\path [-latex](E) edge node[left] {} (F);
    	\path [-latex](F) edge node[left] {} (A);
    	\path [-latex](A) edge node[above] {} (B);
    	\path [-latex](D) edge node[above] {} (G);
    	\path [-latex](B) edge node[above] {} (C);
      \end{tikzpicture}
      \begin{tikzpicture}[scale=2.5]
          	\draw[-] (0,0) -- (2,0) ; 
          	\foreach \x in  {0,1,2} 
          	\draw[shift={(\x,0)},color=black] (0pt,1pt) -- (0pt,-1pt);
          	\foreach \x in {0,1,2} 
          	\draw[shift={(\x,0)},color=black] (0pt,0pt) -- (0pt,0pt) node[above] 
          	{$\x$};
          	 
          	\draw[|-|] (0,-.2) -- (0.121073,-.2);
          	\draw[] (0.06,-.2)node[below] {$J_1$};
          	\draw[line width=.6mm]  (0,-.2) -- (0.121073,-.2);  
          	\draw[|-|] (0.115543,-.4) -- (0.3581,-.4);
          	\draw[] (0.23,-.4)node[below] {$J_2$};
          	\draw[line width=.6mm]  (0.115543,-.4) -- (0.3581,-.4);   
          	\draw[|-|] (0.5,-.2) -- (0.744166,-.2);
          	\draw[] (0.61,-.2)node[below] {$J_3$};
          	\draw[line width=.6mm]  (0.5,-.2) -- (0.744166,-.2);  
          	\draw[|-|] (0.712721,-.4) -- (1.25583,-.4);
          	\draw[] (1,-.4)node[below] {$J_4$};
          	\draw[line width=.6mm]  (0.712721,-.4) -- (1.25583,-.4); 
          	\draw[|-|] (1.14451,-.2) -- (1.6419,-.2);
          	\draw[] (1.4,-.2)node[below] {$J_5$};
          	\draw[line width=.6mm]  (1.14451,-.2) -- (1.6419,-.2);  
          	\draw[|-|] (1.63822,-.4) -- (1.87893,-.4);
          	\draw[] (1.7,-.4)node[below] {$J_6$};
          	\draw[line width=.6mm]  (1.63822,-.4) -- (1.87893,-.4); 
          	\draw[|-|] (1.88901,-.2) -- (2,-.2);
          	\draw[] (1.94,-.2)node[below] {$J_7$};
          	\draw[line width=.6mm]  (1.88901,-.2)
          	--(2,-.2);
          	\draw[|-|] (0,-.7) -- (0.3581,-.7);
          	\draw[] (1,-.7)node[below] {$J$};
          	\draw[line width=.6mm] (0,-.7) -- (0.3581,-.7);   
          	\draw[|-|] (0.5,-.7) -- (1.87893,-.7);
          	\draw[line width=.6mm] ((0.5,-.7) -- (1.87893,-.7); 
          	\draw[|-|] (1.88901,-.7) -- (2,-.7);
          	\draw[line width=.6mm]  (1.88901,-.7)
          	--(2,-.7);          	          		   
          	      	     	           	    	         	      			
          	\end{tikzpicture}
                \caption{Example of bracketing intervals for the cycle
                  graph $C_6$ with one pendant
                  edge.\label{fig:cycle.deco}}
    \end{figure} 

\end{example}

%
\section{Periodic graphs}
\label{sec:periodic}

We begin recalling the definition and some useful facts concerning periodic graphs,
discrete Floquet theory and its relation to the vector potential.
%

\subsection{Periodic graphs and fundamental domains}

The preceding two sections refer to finite graphs.  We consider here
certain classes of infinite graphs, namely $\Gamma$-periodic graphs,
where $\Gamma=\left\langle g_1, g_2, \dots, g_r \right\rangle $ is a
finitely generated and Abelian group with generators $g_1, g_2, \dots,
g_r$. In crystallography, one typically considers $\Gamma=\Z^r$
(see~\cite[Sec.~6.2]{sunada:13}), with generators given,
for example, by $g_1=\left( 1,0,\dots,0\right),g_2=\left(
  0,1,\dots,0\right),\dots,g_r=\left( 0,0,\dots,1\right)$.  We say
that a graph $\Gt$ is \emph{$\Gamma$-periodic} if there is a
free and transitive action of $\Gamma$ on $\Gt$ with compact quotient
$\G=\Gt /\Gamma$ and which is orientation
preserving, i.e., $\Gamma$ acts both on $V$ and $E$ such that
\begin{equation*}
  \bd_+(\gamma e) = \gamma (\bd_+e)
  \quadtext{and}
  \bd_-(\gamma e) = \gamma (\bd_-e)
  \qquad\text{for all $\gamma\in\Gamma$ and $e \in E$.}
\end{equation*}
To avoid trivial situations we assume that the periodic graph $\Gt$ is connected.
As we use the multiplicative notation for the action, we also write
$\Gamma$ multiplicatively.
In particular, we have
\begin{equation*}
  E_{\gamma v} =\gamma E_v
  \qquad\text{for all $\gamma\in\Gamma$ and $e \in E$.}
\end{equation*}
A $\Gamma$-periodic graph $\Gt$ can also be seen as a covering
(see~\cite[Ch.~5 and~6]{sunada:13} or~\cite{sunada:08} for more
details):
\begin{equation*}
  \pi\colon \Gt \rightarrow \G=\Gt /\Gamma.
\end{equation*}

We say that a weighted graph $\Wt =(\Gt,\m)$ is \emph{$\Gamma$-periodic}
if $\Gt$ is $\Gamma$-periodic and if the action of $\Gamma$ on $\Gt$
preserves the corresponding weights, i.e., if
\begin{equation*}
  \m(\gamma v)=\m(v) 
  \quad\text{for all $v \in V$}
  \qquadtext{and}
  \m_{\gamma e}=\m_e
  \quad\text{for all $e \in E$ and $\gamma\in\Gamma$.}
\end{equation*}
Note that the standard or combinatorial weights on a discrete graph
satisfy these conditions automatically.
A $\Gamma$-periodic weighted graph $\Wt =(\Gt ,\m)$ naturally induces a weight $m$ 
on the quotient graph, given by $\m\circ \pi^{-1}$. Notice that this weight 
is well-defined since $\m$ is $\Gamma$-invariant. We will denote this
weight on the quotient simply by $m$.
    
We consider first the following convenient notation adapted to the description
of periodic graphs (see, e.g.,~\cite[Sec.~7]{lledo-post:08}) 
and the important notion of an edge index (see 
\cite[Subsections~1.2 and 1.3]{koro-sabu:14} 
\begin{definition}
  \label{def:fund.dom}
  Let $\Gt=(V,E,\bd)$ be a $\Gamma$-periodic graph.
  \begin{enumerate}
    \item
      \label{fund.dom.a}
      A \emph{vertex}, respectively \emph{edge} \emph{fundamental domain} on a $\Gamma$-periodic graph is
      given by two subsets $\D^V \subset V$ and $\D^E\subset E$
      satisfying
      \begin{align*}
        V(\Gt)&=\bigcup_{\gamma\in\Gamma}\gamma\D^V \quadtext{and}
        \gamma_1 \D^V \cap \gamma_2\D^V=\emptyset
        \quad\text{if $\gamma_1\neq \gamma_2$,}\\
        E(\Gt)&=\bigcup_{\gamma\in\Gamma}\gamma\D^E \quadtext{and}
        \gamma_1 \D^E \cap \gamma_2\D^E=\emptyset \quad\text{if
          $\gamma_1\neq \gamma_2$}
      \end{align*}
      with $\D^E \cap E(V \setminus \D^V)=\emptyset$ (i.e., an edge in
      $\D^E$ has at least one endpoint in $\D^V$).  We often simply
      write $\D$ for a fundamental domain, where $\D$ stands either
      for $\D^V$ or $\D^E$.
      
    \item
      \label{fund.dom.b}
      A \emph{(graph) fundamental domain} of a periodic graph $\Gt$ is
      a partial subgraph
      \begin{equation*}
        \subG =(\D^V,\D^E,\bd \restriction_{\D^E}),
      \end{equation*}
      where $\D^V$ and $\D^E$ are vertex and edge fundamental
      domains, respectively. We call
      \begin{equation*}
        B(\subG,\Gt):=E(\D^V,V \setminus \D^V)
      \end{equation*}
      the set of \emph{connecting edges} of the fundamental domain
      $\subG$ in $\Gt$.\footnote{\label{fn:bridges}
        In~\cite{koro-sabu:14}, Korotyaev and Saburova used the term
        \emph{bridge} for \emph{all} edges connecting a fundamental
        domain with another (non-trivial) translate of the fundamental
        domain.  Korotyaev and Saburova have hence twice as many such
        edges as we have in $B(\subG,\G)$.  Although the name
        ``bridge'' is quite intuitive, it is already used in graph
        theory in a different context; namely for an edge that
        disconnects a graph if it is removed.}
    \end{enumerate}
\end{definition}

\begin{remark}
  \label{rem:coord}
  \indent
  \begin{enumerate}
  \item
  \label{coord.a}
    Note that once a fundamental domain $\D^V$ has been specified
    in a $\Gamma$-periodic graph $\Gt$, we can write any $v\in V(\Gt)$
    uniquely as $v=\xi(v)v_0$ for a unique pair $(\xi(v), v_0) \in
    \Gamma \times \D^V$.  This follows from the fact that the action
    is free and transitive.  We call $\xi(v)$ the
    \emph{$\Gamma$-coordinate of $v$} (with respect to the fundamental
    domain $\D^V$).  Similarly we can define the coordinates for the
    edges: any $e\in E(\Gt)$ can be written as $e=\xi(e)e_0$ for a
    unique pair $(\xi(e), e_0) \in \Gamma \times \D^E$.
    In particular, we have
    \begin{equation*}
      \xi(\gamma v)=\gamma \xi(v)
      \qquadtext{and}
       \xi(\gamma e)=\gamma \xi(e).
    \end{equation*}
  \item
    \label{coord.b}
    Once we have chosen a fundamental domain $\subG=(\D^V,\D^E,\bd)$, 
    we can embed $\subG$
    it into the quotient $\G=\Gt/\Gamma$ of the covering $\pi \colon
    \Gt \rightarrow \G=\Gt / \Gamma$ by
    \begin{equation*}
      \D^V \rightarrow V(\G)=V/\Gamma, \quad
      v \mapsto [v]
      \qquadtext{and}
      \D^E \rightarrow E(\G)=E/\Gamma, \quad
      e \mapsto [e],
    \end{equation*}
    where $[v]$ and $[e]$ denote the $\Gamma$-orbits of $v$ and $e$,
    respectively.  By definition of a fundamental domain, these maps
    are bijective.  Moreover, if $\bd_\pm e=v$ in $\subG$, then also
    $\bd_\pm ([e]) = [v]$ in $\G$, i.e., the embedding is
    a (partial) graph homomorphism. 
  \end{enumerate}
\end{remark}

\begin{definition}
  \label{def:index}
  Let $\Gt=(V,E,\bd)$ be a $\Gamma$-periodic graph with fundamental
  graph $\subG=(\D^V,\D^E,\bd)$.  We define the \emph{index} of an
  edge $e \in E$ as
  \begin{equation*}
    \ind_\subG(e) := \xi(\bd_+e)\left(\xi(\bd_-e)\right)^{-1} \in \Gamma.
  \end{equation*}
\end{definition}

In particular, we have $\ind_\subG \colon E \mapsto \Gamma$, and
$\ind_\subG(e) \ne 1_\Gamma$ iff $e \in \bigcup_{\gamma \in \Gamma}
\gamma B(\subG,\Gt)$, i.e., the index is only non-trivial on the
(translates of the) connecting edges. Moreover, the set of indices and
its inverses generate the group $\Gamma$.

Since the index fulfils $\ind_\subG(\gamma e)=\gamma \ind_\subG(e)$ for all $\gamma
\in \Gamma$ by Remark~\ref{rem:coord}~\eqref{coord.a}, we can extend
the definition to the quotient $\G=\Gt/\Gamma$ by setting
$\ind_{\G}([e])=\ind_\subG(e) $ for all $e\in E(\Gt)$. We denote
also $[B(\subG,\Gt)]:=\{ [e] \mid e\in B(\subG,\Gt)\}$.

\subsection{Discrete Floquet theory}

Let $\Wt=(V,E,\bd,\m)$ be a weighted $\Gamma$-periodic graph and
fundamental domain $\subG=(\D^V,\D^E,\bd)$ with corresponding weights
inherited from $\Wt$. In this context one has the natural Hilbert space
identifications
\begin{equation*}
 \ell_2(V,\m)
 \cong \ell_2(\Gamma) \otimes \ell_2(D^V,m)
 \cong \ell_2\bigl(\Gamma, \ell_2(\D^V,m)\bigr).
\end{equation*}
Roughly speaking, a discrete Floquet transformation is a partial
Fourier transformation which is applied only on the group part, i.e.,
\begin{equation*}
  F \colon \ell_2(\Gamma)\rightarrow L_2(\widehat \Gamma),
  \qquad
  \left(F \mathbf a\right) \left(\chi\right)
  :=\sum_{\gamma\in\Gamma} \overline{\chi(\gamma)} a_\gamma
\end{equation*}
for $\mathbf a=\left\lbrace a_\gamma \right\rbrace_{\gamma\in\Gamma}
\in \ell_2(\Gamma)$ and where $\widehat \Gamma$ denotes the character
group of $\Gamma$.  We adapt to the discrete context of graphs the
main results concerning Floquet theory needed later. We refer to
\cite[Section~3]{lledo-post:07} as well as~\cite{koro-sabu:14} for
details and additional motivation.

For any character $\chi\in\widehat{\Gamma}$ consider the space of \emph{equivariant functions} on vertices and edges
\begin{eqnarray*}
  \ell_2^\chi(V,m)
  &:=&\left\lbrace g\colon V \rightarrow \C \mid
    g(\gamma v)= \chi(\gamma)g(v) 
    \text{ for all  } v\in V \text{ and } \gamma\in\Gamma\right\rbrace, \\
    \ell_2^\chi(E,m)
  &:=&\left\lbrace \eta\colon E\rightarrow \C \mid
    \eta_{\gamma e}= \chi(\gamma)\eta_e 
    \text{ for all  } e\in E \text{ and } \gamma\in\Gamma\right\rbrace.
\end{eqnarray*}
These spaces have the natural inner product:
\begin{equation*}
  \left\langle g_1,g_2\right\rangle
  :=\sum_{v\in \D^V} g_1(v) \overline{g_2(v)}  m(v)
\end{equation*}
for a fundamental domain $\D^V$ (and similarly for the equivariant
scalar product on $D^E$).  Note that the definition of the inner
product is independent of the choice of fundamental domain (due to the
equivariance).  The following decomposition result is standard, see,
for example,~\cite{koro-sabu:14} or~\cite{higuchi-shirai:99}.

\begin{proposition}
  \label{prp:floq-th}
  Let $\Wt=(\Gt,\m)$  be a periodic weighted graph with $\Gt=(V,E,\bd)$.  Then
  there is a unitary transformation
  \begin{equation*}
    \Phi\colon\ell_2(V(\Gt))\rightarrow 
    \int_{\widehat \Gamma }^\oplus \ell_2^\chi(V,m) \dd \chi
    \qquadtext{given by}
    \left( \Phi f \right)_\chi(v)
    =\sum_{\gamma\in\Gamma}\overline{\chi({\gamma})} f(\gamma v)
  \end{equation*}
  such that
  \begin{equation*}
    \sigma\bigl(\Delta^{\!\Wt} \bigr)
    = \bigcup_{\chi \in \widehat \Gamma} 
       \sigma \bigl(\eqLapl \bigr).
  \end{equation*}
  where as $\eqLapl :=\Delta^{\!\Wt}
  \restriction_{\ell_2^\chi(V,m)}$
    denotes the \emph{equivariant Laplacian}.
\end{proposition}
The equivariant Laplacian may also be described in terms of a first order
approach by defining $d^\chi\colon\ell_2^\chi(V,m)\to \ell_2^\chi(E,m)$
just by restriction of $d$ to the subspace $\ell_2^\chi(V,m)$:
\begin{equation*}
  (d^\chi g)_e = g(\bd_+ e) - g(\bd_-e)\;,\quad g\in\ell_2^\chi(V,m)\;.
\end{equation*}
It is straightforward to check
that $d^\chi g \in \ell_2^\chi(E,m)$ if $g \in \ell_2^\chi(V,m)$ and that
$\eqLapl=(d^\chi)^* d^\chi$.

\subsection{Vector potential as a Floquet parameter}

The following result shows that in the case of Abelian groups $\Gamma$
the vector potential can be interpreted as a Floquet parameter of the
periodic graph $\Gt\to\G$ (see Remark~\ref{rem:coord}~(b)).  Consider
the following unitary maps (see also~\cite{kos:89} for manifolds):
\begin{align*}
  U^V &\colon\ell_2(V(\G),m) \to \ell_2^\chi(V,m),&
  \bigl(U^Vf\bigr)( v)
  &= \chi(\xi(v))f([v]),\\
  U^E &\colon\ell_2(E(\G),m) \to \ell_2^\chi(E,m),&
  \bigl(U^E \eta \bigr)_{ e}
  &= \chi(\xi(e)) \left( \eta  \right)_{[e]}.
\end{align*}
It is straightforward to see that $U^V f\in \ell_2^\chi (V,m)$ for all
$f\in \ell_2(V(\G),m)$
 and that
$U^V$ is unitary (similarly for $U^E$).

\begin{definition}
  Let $\Wt=(\Gt,\m)$ be a $\Gamma$-periodic weighted graph with finite
  quotient $\W=(\G,m)$ and $\subG$ be a fundamental domain.  If
  $\alpha$ is a vector potential acting on $\G$, we say that $\alpha$
  has \emph{the lifting property} if there exists
  $\chi\in\widehat{\Gamma}$ such that:
  \begin{equation}
    \e^{i\alpha_{[e]}}=\chi\left( \ind_\subG(e)\right) 
    \quad \text{ for all } e \in E(\Gt).
  \end{equation}
  We denote the set of all the vector potentials with the lifting
  property as $\mathcal{A}_\subG$.

\end{definition}

\begin{proposition}
  \label{prp:floq.mag}
  \indent
  Let $\Wt=(\Gt,\m)$ be a $\Gamma$-periodic weighted graph with finite quotient $\W=(\G,m)$ and $\subG$ be a fundamental domain, then
\begin{equation}\label{floq.mag.a}
 \sigma(\Delta^{\!\Wt})= \bigcup_{\alpha\in \mathcal{A}_\subG} \sigma(\Delta^{\W}_{\alpha})
 \subset [0,2p_\infty] \setminus \MSG^{\W}.
\end{equation}
\end{proposition}
\begin{proof}
  By Proposition~\ref{prp:floq-th}, it is enough to show
    \begin{equation*}
     \bigcup_{\chi \in \widehat \Gamma} 
            \sigma \bigl(\eqLapl \bigr)
      = \bigcup_{\alpha\in \mathcal{A}_\subG} \sigma\bigl(\Delta^{\W}_{\alpha}\bigr)
    \end{equation*}  
    ``$\subset$'': Consider a character $\chi \in \widehat{\Gamma}$
    and define a vector potential on $\G$ as follows
    \begin{equation}
      \label{eq:rel.alpha.chi}
      \e^{\im  \alpha_{[e]}}=\chi(\ind_\subG(e))\;,\quad e\in E\;.
    \end{equation}
  Then we have
  \begin{equation*}
    (d^\chi U^Vf)_e
    = (U^Vf)(\bd_+e)-(U^Vf)(\bd_-e)
    = \chi(\xi(\bd_+e))f([\bd_+e]) - \chi(\xi(\bd_-e))f([ \bd_-e]).
  \end{equation*}
  On the other hand, we have
  \begin{equation*}
    (U^E d_\alpha f)_e
    = \chi(\xi(e))
      \Bigl(\e^{\im \alpha_{[e]}/2} f([ \bd_+e])-\e^{-\im \alpha_{[e]}/2} f([ \bd_-e])\Bigr).
  \end{equation*}
  Therefore, the intertwining equation $d^\chi U^V= U^E d_\alpha$ holds if
  \begin{equation*}
    \chi(\xi(\bd_+e))=\chi(\xi(e))\e^{\im \alpha_{[e]}/2}
    \quadtext{and}
    \chi(\xi(\bd_-e))=\chi(\xi(e))\e^{-\im \alpha_{[e]}/2}
  \end{equation*}
  or, equivalently, if 
  \begin{equation*}
    \chi(\xi(\bd_+e))\e^{-\im \alpha_{[e]}/2}
    = \chi(\xi(\bd_-e)) \e^{\im \alpha_{[e]}/2}
  \end{equation*}
  or
  \begin{equation*}
    \e^{\im \alpha_{[e]}} 
    = \chi(\xi(\bd_+e))\chi(\xi(\bd_-e))^{-1}
    = \chi(\ind_\subG(e))\;.
  \end{equation*}
  But this equation is true by definition of the vector potential on $\G$ given in Equation~\eqref{eq:rel.alpha.chi}.
  Finally, since $\eqLapl=(d^\chi)^*d^\chi$ and $\Delta^{\W}_\alpha=d_\alpha^*d_\alpha$ it is clear that these Laplacians
  are unitarily equivalent.

  ``$\supset$'': Let $\alpha\in \mathcal{A}_\subG$ and $E_\subG\subset
  E(\G)$ is such that $\{\ind_\subG(e)|[e]\in E_\subG\}$ is a basis of
  the group $\Gamma$.  Then define
    \begin{equation}
      \label{eq:rel.alpha.chi2}
      \chi(\ind_\subG(e))=\e^{\im  \alpha_{[e]}}\;,\quad e\in E_\subG\;.
    \end{equation}
    so we can extend $\chi$ to all $\Gamma$, so $\chi\in
    \widehat{\Gamma}$. As before, we can show $\sigma \bigl(\eqLapl
    \bigr) = \sigma\left(\Delta^{\W}_{\alpha}\right).$
\end{proof}

\begin{remark}
  \indent
  \begin{enumerate}
    \label{remark:z-periodic}
  \item If $\Gt \to\G$ is a maximal Abelian covering, then we have
    $\sigma(\Delta^{\Gt})= [0,2p_\infty] \setminus \MSG^{\G}$. In
    particular, this is true if $\Gt$ is a tree.
  \item If $\Gamma=\Z$ and if each fundamental domain is connected to
    its neighbours by a single connecting edge, i.e.,
    $|B(\subG,\Gt)|=1$ then we have the following situation: Define
    the vector potential $\alpha^t$ on $\G$ as $\alpha^t=t$ if $[e]\in
    [B(\subG,\Gt)]$ and zero otherwise. Denote
    $\sigma\left(\Delta^{\Gt}_{\alpha^t}\right):=\{ \lambda^t_i \mid
    i=1,\dots,n\}$ be the eigenvalues in ascending order and repeated
    according to their multiplicities, then following results
    in~\cite{ekw:10} we obtain
  \begin{equation*}
    \sigma(\Delta^{\Gt})
    =\bigcup_{i=1}^{|V(\G)|} \bigl[ \min\bigl\lbrace \lambda^0_i,\lambda^\pi_i 
                       \bigr\rbrace , 
           \max\bigl\lbrace \lambda^0_i,\lambda^\pi_i \bigr\rbrace  \bigr].
  \end{equation*} 
 \end{enumerate}   
\end{remark}

Let $\W=(\G,m)$ a weighted graph. We say that $\W$ has the \emph{Full
  spectrum property (FSP)} if $\Delta^\W=[0,2\rho_\infty]$. In fact,
$\W$ has the FSP iff $\SG^\W=\emptyset$. The next conjecture is stated
in~\cite{higuchi-shirai:99}. Let $\Gt$ be a maximal Abelian covering
of $\G$, if $\G$ has no vertex of degree $1$ then $\Gt$ has the
FSP. Moreover, Higuchi and Shirai propose the next problem:
Characterise all finite graphs whose maximal Abelian covering do not
has the FSP.  Here, we partially solve the conjecture and the problem.

The following result verifies Higuchi-Shirai's conjecture
in~\cite{higuchi-shirai:04} for $\Z$-periodic trees.

\begin{theorem}
  \label{theo:FSP}
  Let $\Wt=(\Gt,\m)$ be a $\Z$-periodic tree with standard or
  combinatorial weights and quotient graph $\W=(\G,m)$.  Then
  the following conditions are equivalent:
  \begin{enumerate}
  \item $\Wt$ has the full spectrum property;
  \item $\SG^{\Wt}=\emptyset$;
  \item $\Gt$ is the lattice $\Z$;
  \item $\MSG^{\W}=\emptyset$;
  \item $\G$ is a cycle graph;
  \item $\G$ has no vertex of degree $1$.
  \end{enumerate}
\end{theorem}
\begin{proof}
  In this situation $\Gt\to\G$ is a maximal Abelian covering if and
  only if the Betti number $\Betti(\G)=1$ if and only if $\Gt$ is a
  $\Z$-periodic tree.  Since $\Gt$ is a tree and $\Gt\to\G$ is a
  maximal Abelian covering, then we have $\sigma(\Delta^{\!\Wt})=
  [0,2p_\infty] \setminus \MSG^{\W}$. The result then follows by
  Corollary~\ref{cor:spectral_gap}.
\end{proof}

\begin{remark}
  Example~\ref{exa:propy} confirms the previous theorem. In addition,
  if $\Betti(\G)\geq 2$ one can easily produce periodic graphs
  based, e.g., on Example~\ref{exa:cycle-edge} that do not having the
  full spectrum property.
\end{remark}

\subsection{Discrete bracketing technique}
We apply now the technique stated in Proposition~\ref{prp:key-obs} to periodic graphs.
  
\begin{theorem}
  \label{theo:main}
  Let $\Wt=(\Gt,\m)$ a $\Gamma$-periodic graph and
  $\pi\colon\Gt\rightarrow \G=\Gt/\Gamma$ with fundamental domain
  $\subG=(\D^V,\D^E,\bd)$.  We let
  \begin{equation*}
    E_0 := [B(\subG,\Gt)]
  \end{equation*}
  be the image of the connectivity edges on the quotient and $V_0$ in
  the neighbourhood of $E_0$. Define by
  \begin{equation*}
    \EG := \G - E_0
    \qquadtext{and}
    \VG := \G - V_0.
  \end{equation*}
  the corresponding edge and vertex virtualised partial graphs,
  respectively. Then
  \begin{equation*}
    \sigma(\Delta^{\!\Wt}) 
    \subset  \bigcup_{k=1}^{|V(\G)|}
    \underbrace{[\lambda_k(\Delta^{\EG}),\lambda_k(\Delta^{\VG})]}_{=:J_k}
  \end{equation*}
  where $\sigma(\Delta^{\EG})$ and $\sigma(\Delta^{\VG})$ are as in
  Theorem~\ref{thm:technique}, i.e., the eigenvalues are written in
  ascending order and repeated according to their multiplicities.
\end{theorem}
\begin{proof}
  By Proposition~\ref{prp:floq.mag} we have
  \begin{equation*}
    \sigma(\Delta^{\!\Wt}) = \bigcup_{\alpha\in \mathcal{A}_\subG}
    \sigma(\Delta^{\W}_{\alpha}).
  \end{equation*} 
  Now, by the bracketing technique of Propositions~\ref{VirtualEdges}
  and~\ref{DMDL}, we have for any potential $\alpha\in
  \mathcal{A}_\subG$,
  \begin{equation*}
    \lambda_k(\Delta^{\EG}) 
    \le \lambda_k(\Delta^{\G}_{\alpha}) 
    \le \lambda_k(\Delta^{\VG})
    \quadtext{for all}
    k=1,\dots,|V(\G)|
  \end{equation*}
  provided the vector potential is supported only on the (translates
  of the) connecting edges $E_0 \subset E(\G)$.  Therefore
  \begin{equation*}
    \sigma(\Delta^{\!\Wt}) 
    \subset  \bigcup_{k=1}^{|V(\G)|}
    \underbrace{[\lambda_k(\Delta^{\EG}),\lambda_k(\Delta^{\VG})]}_{=:J_k}.  
    \qedhere
  \end{equation*}
\end{proof}

Note that the bracketing intervals $J_k$ depend on the fundamental
domain $\subG$.  A good choice would be one where the set of
connecting edges is as small as possible.  In this case, we have a
good chance that the bracketing intervals $J_k$ actually do not cover
the full interval $[0,2\rho_\infty]$.  This is geometrically a
``thin--thick'' decomposition, just as in~\cite{lledo-post:08b}, where
a fundamental domain only has a few connections to the complement.

%
%
\section{Examples and applications: Spectral gaps for periodic graphs}
\label{sec:examples}
%

We conclude this article with several applications of the methods
developed before.  Our first example confirms the results by Suzuki
in~\cite{suzuki:13} on periodic graphs with pendants edges using our
simple geometric method (Example~\ref{exa:suzuki}).  The other
examples are more elaborate and include an idealised model of
polypropylene (Example~\ref{exa:propy}) and polyacetylene molecules
(Example~\ref{exa:poly}).  The last example can be understood as an
intermediate covering of graphane and shows that the bracketing
technique developed here extends to more general situation than the
$\Z$-periodic trees, i.e., to higher Betti numbers of the quotient
graphs.

\begin{example}[Suzuki's example]
  \label{exa:suzuki}
  Consider the graph $T_n$ consisting of the $\Z$-lattice and a
  pendant edge at every $n$-th vertex as decoration (see
  Figure~\ref{fig:suzuki.tn}) with standard weights as
  in~\cite{suzuki:13}. It is easy to see that the tree $T_n$ is
  the maximal Abelian covering graph of $\G_n$, where $\G_n$ is just the
  cycle graph $C_n$ decorated with an additional edge attached to some
  vertex of the cycle (for example see Figure~\ref{fig:cycle.deco} for
  $\G_6$) with the standard weights. The graph $T_n$ has spectral
  gaps, i.e., $\SG^{T_n}\ne \emptyset$, as Suzuki proves.  Our
  analysis allows an alternative and short proof (based on the
  criterion in Corollary~\ref{cor:spectral_gap}): As $T_n$ is a tree,
  we can apply Theorem~\ref{theo:FSP}: as $\G_n$ has a vertex of
  degree $1$, the covering graph $T_n$ cannot have the full spectrum
  property.
\end{example}

\begin{figure}[h]
  \centering
  \begin{tikzpicture}[auto,
    vertex/.style={circle,draw=black!100,fill=black!100, thick,
      inner sep=0pt,minimum size=1mm},scale=4]
    
    \node (C) at (-.35,0) [vertex,inner sep=.25pt,minimum size=.25pt,label=above:]{};
    \node (D) at (-.5,0) [vertex,inner sep=.25pt,minimum size=.25pt,label=above:]{};
    \node (E) at (-.65,0) [vertex,inner sep=.25pt,minimum size=.25pt,label=above:]{};
  	                
    \node (Z1) at (-.2,0) [vertex,label=below:$v_{-1}$]{};    
    \node (A) at (0,0) [vertex,label=below:$v_0$]{};
    \node (B) at (.2,0) [vertex,label=below:$v_1$]{};
    \node (C) at (.35,0) [vertex,inner sep=.25pt,minimum size=.25pt,label=above:]{};
    \node (D) at (.5,0) [vertex,inner sep=.25pt,minimum size=.25pt,label=above:]{};
    \node (E) at (.65,0) [vertex,inner sep=.25pt,minimum size=.25pt,label=above:]{};
    \node (F) at (.8,0) [vertex,label=below:$v_{n-1}$]{};               
    \node (G) at (1,0) [vertex,label=below:]{}; 
    \node (A1) at (0,.5) [vertex,label=above:$w_0$]{};
    \node (B1) at (1,.5) [vertex,label=above:$w_1$]{};
    \draw[-latex] (A) to[] node[above] {} (A1);
    \draw[-latex] (Z1) to[]  node[above] {} (A);   
    \draw[-latex] (A) to[]  node[above] {} (B); 
    \draw[-latex] (F) to[]  node[above] {} (G); 
    \draw[-latex] (G) to[]  node[above] {} (B1); 
    \node (A) at (1,0) [vertex,label=below:$v_n$]{};  
    \node (B) at (1.2,0) [vertex,label=below:$v_{n+1}$]{};
    \node (C) at (1.35,0) [vertex,inner sep=.25pt,minimum size=.25pt,label=above:]{};
    \node (D) at (1.5,0) [vertex,inner sep=.25pt,minimum size=.25pt,label=above:]{};
    \node (E) at (1.65,0) [vertex,inner sep=.25pt,minimum size=.25pt,label=above:]{};
    \node (F) at (1.8,0) [vertex,label=below:$v_{2n-1}$]{};               
    \node (G) at (2,0) [vertex,label=below:$v_{2n}$]{};              
    \node (B1) at (2,.5) [vertex,label=above:$w_2$]{};
    \node (Z2) at (2.2,0) [vertex,label=below:$v_{2n+1}$]{}; 
    \draw[-latex] (G) to[]  node[above] {} (Z2);   
    \draw[-latex] (A) to[]  node[above] {} (B); 
    \draw[-latex] (F) to[]  node[above] {} (G); 
    \draw[-latex] (G) to[]  node[above] {} (B1);  	 
    \node (C) at (2.35,0) [vertex,inner sep=.25pt,minimum size=.25pt,label=above:]{};
    \node (D) at (2.5,0) [vertex,inner sep=.25pt,minimum size=.25pt,label=above:]{};
    \node (E) at (2.65,0) [vertex,inner sep=.25pt,minimum size=.25pt,label=above:]{};
  \end{tikzpicture}
  \caption{The infinite tree graph $T_n$ with a pendant vertex every
    $n$ vertices along the $\Z$-lattice, see~\cite{suzuki:13}.}
\label{fig:suzuki.tn}
\end{figure}
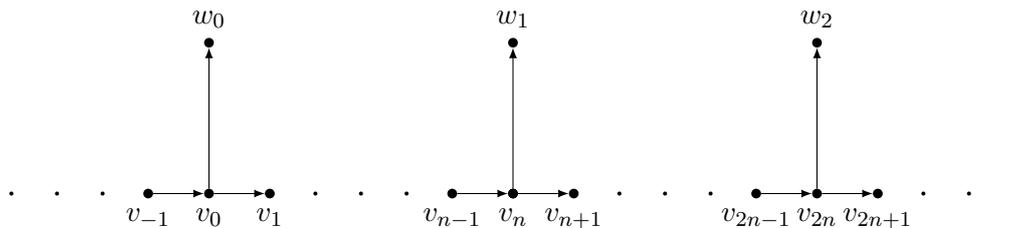

\begin{example}[Polypropylene]
  \label{exa:propy}
  Consider the graph associated to a thermoplastic polymer, the
  \emph{polypropylene}.  This structure consists of a sequence of
  carbon atoms (white vertices) with hydrogen (black vertices) and the
  methyl group $\mathrm{CH}_3$.  We choose the infinite covering graph
  $\Gt$ as an idealised model of polypropylene (see
  Figure~\ref{subfig:PP}); this graph is a covering graph of the
  bipartite graph denoted as $\G$ (see Figure~\ref{subfig:PP_0}).
  Again, by Corollary~\ref{cor:spectral_gap} we get that the set of
  magnetic spectral gaps is not empty $\MSG^{\G}\ne 0$ and by
  Proposition~\ref{prp:floq.mag} we conclude that the Laplacian on
  $\Gt$ has spectral gaps. We show how to apply the technique
  developed in this article twice:
  
  First, let $E_0=\{e_1\}$. Then $V_0=\{v_1\}$ is in the neighbourhood
  of $E_0$ (see Definition~\ref{def:admissible}). Using the notation
  in Theorem~\ref{thm:technique} and Proposition~\ref{prp:floq.mag} we
  get $\sigma(\Delta^{\Gt})\subset J$, where $J$ is a subset of
  $[0,2]$ (see Figure~\ref{subfig:spectrum}).  Since $\G$ is bipartite
  we obtain $\sigma(\Delta^{\Gt})\subset \kappa(J)$ by
  Proposition~\ref{prp:bipartite.sym}.  Therefore, the symmetry gives
  tighter localisation of the spectrum $\sigma(\Delta^{\Gt})\subset
  J\cap \kappa(J)$.
  
  But the set $V_0'=\{v_2\}$ is also in the neighbourhood of the
  previous set $E_0$. Then apply the same argument as before for $E_0$
  and $V_0'$ to obtain $J', \kappa(J')$. Again we obtain
  $\sigma(\Delta^{\Gt})\subset J'\cap\kappa(J')$.  We conclude that
  $\sigma(\Delta^{\Gt})\subset J\cap\kappa(J)\cap
  J'\cap\kappa(J')$. In fact, in Figure~\ref{subfig:spectrum} we see
  that our technique gives a very good estimation for the spectrum.
\end{example}

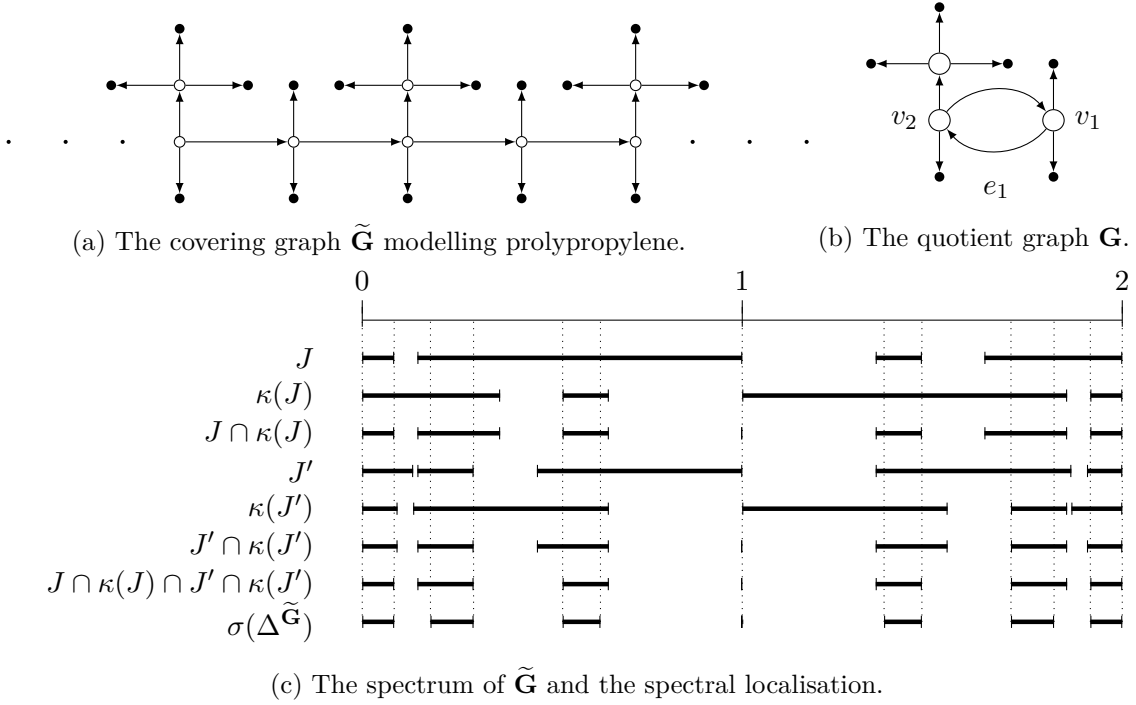
\begin{figure}[h]
   	\centering
\subcaptionbox{The covering graph $\Gt$ modelling prolypropylene. \label{subfig:PP}}%
  [.6\linewidth]{\begin{tikzpicture}[auto, vertex/.style={circle,draw=black!100,fill=black!100, thick,
       	                    inner sep=0pt,minimum size=1mm},scale=3]
       	        \node (D) at (-1.25,0) [vertex,inner sep=.25pt,minimum size=.25pt,label=above:]{};
       	        \node (D) at (-1,0) [vertex,inner sep=.25pt,minimum size=.25pt,label=above:]{};
       	        \node (D) at (-.75,0) [vertex,inner sep=.25pt,minimum size=.25pt,label=above:]{};
       	        \node (O) at (-.5,0) [circle, minimum width=4pt, draw, inner sep=0pt] {};                
       	        \node (A) at (0,0) [circle, minimum width=4pt, draw, inner sep=0pt] {};
       	        \node (B) at (.5,0) [circle, minimum width=4pt, draw, inner sep=0pt]{};
       	        \node (C) at (1,0) [circle, minimum width=4pt, draw, inner sep=0pt]{};
       	        \node (D) at (1.5,0) [circle, minimum width=4pt, draw, inner sep=0pt]{};
       	        \draw[-latex] (O) to[] node[above] {} (A);
       	        \draw[-latex] (A) to[] node[above] {} (B);
       	    	\draw[-latex] (B) to[]  node[above] {} (C);
       	    	\draw[-latex] (C) to[]  node[above] {} (D);
       	        \node (O1) at (-.5,.25) [circle, minimum width=4pt, draw, inner sep=0pt] {};                
       	        \node (A1) at (0,.25) [vertex,label=left:] {};
       	        \node (B1) at (.5,.25) [circle, minimum width=4pt, draw, inner sep=0pt]{};
       	        \node (C1) at (1,.25) [vertex,label=left:]{};
       	        \node (D1) at (1.5,.25) [circle, minimum width=4pt, draw, inner sep=0pt]{}; 	        
       	        \draw[-latex] (O) to[] node[above] {} (O1);
       	        \draw[-latex] (A) to[] node[above] {} (A1);
       	    	\draw[-latex] (B) to[]  node[above] {} (B1);
       	    	\draw[-latex] (C) to[]  node[above] {} (C1);
       	    	\draw[-latex] (D) to[]  node[above] {} (D1);
       	        \node (O2) at (-.5,-.25) [vertex,label=left:] {};                
       	        \node (A2) at (0,-.25) [vertex,label=left:] {};
       	        \node (B2) at (.5,-.25) [vertex,label=left:]{};
       	        \node (C2) at (1,-.25) [vertex,label=left:]{};
       	        \node (D2) at (1.5,-.25) [vertex,label=left:]{}; 	        
       	        \draw[-latex] (O) to[] node[above] {} (O2);
       	        \draw[-latex] (A) to[] node[above] {} (A2);
       	    	\draw[-latex] (B) to[]  node[above] {} (B2);
       	    	\draw[-latex] (C) to[]  node[above] {} (C2);
       	    	\draw[-latex] (D) to[]  node[above] {} (D2); 	    	
       	        \node (O) at (-.8,.25) [vertex,label=left:] {}; 
       	        \node (A) at (-.5,.5) [vertex,label=left:] {};                 
       	        \node (B) at (.2,.25) [vertex,label=left:]{};
       	        \node (C) at (.5,.5) [vertex,label=left:] {}; 
       	        \node (D) at (1.2,.25) [vertex,label=left:]{};
       	        \node (E) at (1.5,.5) [vertex,label=left:] {}; 	        
       	        \draw[-latex] (O1) to[] node[above] {} (O);
       	        \draw[-latex] (O1) to[] node[above] {} (A);
       	    	\draw[-latex] (B1) to[]  node[above] {} (B);
       	    	\draw[-latex] (B1) to[] node[above] {} (C);
       	    	\draw[-latex] (D1) to[]  node[above] {} (D);
       	    	\draw[-latex] (D1) to[]  node[above] {} (E);
      			\node (O) at (-.2,.25) [vertex,label=left:] {};                
       	        \node (B) at (.8,.25) [vertex,label=left:]{};
       	        \node (D) at (1.8,.25) [vertex,label=left:]{}; 	        
       	        \draw[-latex] (O1) to[] node[above] {} (O);
       	    	\draw[-latex] (B1) to[]  node[above] {} (B);
       	    	\draw[-latex] (D1) to[]  node[above] {} (D);  	    		
       	    \node (D) at (1.75,0) [vertex,inner sep=.25pt,minimum size=.25pt,label=above:]{};
       	    \node (D) at (2,0) [vertex,inner sep=.25pt,minimum size=.25pt,label=above:]{};
       	    \node (D) at (2.25,0) [vertex,inner sep=.25pt,minimum size=.25pt,label=above:]{};
       	    \end{tikzpicture}}
\subcaptionbox{The quotient graph $\G$.\label{subfig:PP_0}}
  [.3\linewidth]{      	    \begin{tikzpicture}[auto, vertex/.style={circle,draw=black!100,fill=black!100, thick,
    	                    inner sep=0pt,minimum size=1mm},scale=3]
    	        \node (O) at (-.5,0) [circle, minimum width=8pt, draw, inner sep=0pt, label=left:{{$v_2$}}]{};                
    	        \node (A) at (0,0) [circle, minimum width=8pt, draw, inner sep=0pt,label=right:{{$v_1$}}]{};
    	        \draw[-latex] (A) to[bend left=50] node[label=below:{$e_1$}] {} (O); 
    	        \draw[-latex] (O) to[bend left=50] node[below] {} (A); 
    	        \node (O1) at (-.5,.25) [circle, minimum width=8pt, draw, inner sep=0pt] {};                
    	        \node (A1) at (0,.25) [vertex,label=left:] {};	        
    	        \draw[-latex] (O) to[] node[above] {} (O1);
    	        \draw[-latex] (A) to[] node[above] {} (A1);
    	        \node (O2) at (-.5,-.25) [vertex,label=left:] {};                
    	        \node (A2) at (0,-.25) [vertex,label=left:] {};	        
    	        \draw[-latex] (O) to[] node[above] {} (O2);
    	        \draw[-latex] (A) to[] node[above] {} (A2);	    	
    	        \node (O) at (-.8,.25) [vertex,label=left:] {}; 
    	        \node (A) at (-.5,.5) [vertex,label=left:] {};                 	        
    	        \draw[-latex] (O1) to[] node[above] {} (O);
    	        \draw[-latex] (O1) to[] node[above] {} (A);
   			\node (O) at (-.2,.25) [vertex,label=left:] {};                 	        
    	        \draw[-latex] (O1) to[] node[above] {} (O); 	   		
    	    \end{tikzpicture}}\\
\centering
  \subcaptionbox{The spectrum of $\Gt$ and the spectral localisation.
    \label{subfig:spectrum}}
  [\linewidth]{	\centering	\begin{tikzpicture}[scale=5]
              	\draw[-] (0,.5) -- (2,0.5) ; 
              	\foreach \x in  {0,1,2} 
              	\draw[shift={(\x,0.5)},color=black] (0pt,1pt) -- (0pt,-.5pt);
              	\foreach \x in {0,1,2} 
              	\draw[shift={(\x,0.5)},color=black] (0pt,0.5pt) -- (0pt,1.5pt) node[above] 
              	{$\x$};
    
    			\draw[] (-.1,.4)node[left] {$J$};
              	\draw[|-|]  (0,.4) -- (.0840,.4);
              	\draw[line width=.6mm]   (0,.4) -- (.0840,.4);
              	\draw[|-|]  (0.145069,.4) -- (1,.4);
              	\draw[line width=.6mm]   (0.145069,.4) -- (1,.4);
              	\draw[|-|]  (1.3511,.4) -- (1.47273,.4);
              	\draw[line width=.6mm]   (1.3511,.4) -- (1.47273,.4);          	
              	\draw[|-|]  (1.63714,.4) -- (2,.4);
              	\draw[line width=.6mm]   (1.63714,.4) -- (2,.4);  
              	
    			\draw[] (-.1,.3)node[left] {$\kappa(J)$};
              	\draw[|-|]  (2-0,.3) -- (2-.0840,.3);
              	\draw[line width=.6mm]   (2-0,.3) -- (2-.0840,.3);
              	\draw[|-|]  (2-0.145069,.3) -- (2-1,.3);
              	\draw[line width=.6mm]   (2-0.145069,.3) -- (2-1,.3);
              	\draw[|-|]  (2-1.3511,.3) -- (2-1.47273,.3);
              	\draw[line width=.6mm]   (2-1.3511,.3) -- (2-1.47273,.3);  
              	\draw[|-|]  (2-1.63714,.3) -- (2-2,.3);
              	\draw[line width=.6mm]   (2-1.63714,.3) -- (2-2,.3);            	
    
    			\draw[] (-.1,.2)node[left] {$J\cap\kappa(J)$};          	
              	\draw[|-|]  (0,.2) -- (.0840,.2);
              	\draw[line width=.6mm]   (0,.2) -- (.0840,.2);
              	\draw[|-|]  (0.145069,.2) -- (0.36286,.2);
              	\draw[line width=.6mm]   (0.145069,.2) -- (0.36286,.2);
              	\draw[|-|]  (0.52727,.2) -- (0.6489,.2);
              	\draw[line width=.6mm]   (0.52727,.2) -- (0.6489,.2);  
              	\draw[|-|]  (1.3511,.2) -- (1.47273,.2);
              	\draw[line width=.6mm]   (1.3511,.2) -- (1.47273,.2); 
              	\draw[|-|]  (1.63714,.2) -- (1.85493,.2);
              	\draw[line width=.6mm]   (1.63714,.2) -- (1.85493,.2); 
              	\draw[|-|]  (1.916,.2) -- (2,.2);          	         	         	
              	\draw[line width=.6mm]   (1.916,.2) -- (2,.2);
              	\draw[|-|] (.9999,.2) -- (1.00001,.2);          	          	    	
    
    			\draw[] (-.1,.1)node[left] {$J'$};           	
              	\draw[|-|]  (0,.1) -- (0.133975,.1);
              	\draw[line width=.6mm]   (0,.1) -- (0.133975,.1);
              	\draw[|-|]  (0.145069,.1) -- (0.292893,.1);
              	\draw[line width=.6mm]   (0.145069,.1) -- (0.292893,.1);
              	\draw[|-|]  (0.459592,.1) -- (1,.1);
              	\draw[line width=.6mm]   (0.459592,.1) -- (1,.1);          	
              	\draw[|-|]  (1.3511,.1) -- (1.86603,.1);
              	\draw[line width=.6mm]   (1.3511,.1) -- (1.86603,.1);          	
              	\draw[|-|]  (1.9071,.1) -- (2,.1);
              	\draw[line width=.6mm]   (1.9071,.1) -- (2,.1);           	
    
    			\draw[] (-.1,0)node[left] {$\kappa(J')$};          	
              	\draw[|-|]  (2-0,0) -- (2-0.133975,0);
              	\draw[line width=.6mm]   (2-0,0) -- (2-0.133975,0);
              	\draw[|-|]  (2-0.145069,0) -- (2-0.292893,0);
              	\draw[line width=.6mm]   (2-0.145069,0) -- (2-0.292893,0);
              	\draw[|-|]  (2-0.459592,0) -- (2-1,0);
              	\draw[line width=.6mm]   (2-0.459592,0) -- (2-1,0);          	
              	\draw[|-|]  (2-1.3511,0) -- (2-1.86603,0);
              	\draw[line width=.6mm]   (2-1.3511,0) -- (2-1.86603,0);          	
              	\draw[|-|]  (2-1.9071,0) -- (2-2,0);
              	\draw[line width=.6mm]   (2-1.9071,0) -- (2-2,0);    
    
    			\draw[] (-.1,-.1)node[left] {$J'\cap\kappa(J')$};     
              	\draw[|-|]  (0,-.1) -- (2-1.9071,-.1);
              	\draw[line width=.6mm]   (0,-.1) -- (2-1.9071,-.1);
              	\draw[|-|]  (0.145069,-.1) -- (0.292893,-.1);
              	\draw[line width=.6mm]   (0.145069,-.1) -- (0.292893,-.1);    
              	\draw[|-|]  (0.459592,-.1) -- (0.648902,-.1);
              	\draw[line width=.6mm]   (0.459592,-.1) -- (0.648902,-.1);     	
              	\draw[|-|]  (2-0,-.1) -- (1.9071,-.1);
              	\draw[line width=.6mm]   (2-0,-.1) -- (1.9071,-.1);
              	\draw[|-|]  (2-0.145069,-.1) -- (2-0.292893,-.1);
              	\draw[line width=.6mm]   (2-0.145069,-.1) -- (2-0.292893,-.1);    
              	\draw[|-|]  (2-0.459592,-.1) -- (2-0.648902,-.1);
              	\draw[line width=.6mm]   (2-0.459592,-.1) -- (2-0.648902,-.1);  
              	\draw[|-|] (.9999,-.1) -- (1.00001,-.1);
              	
    			\draw[] (-.1,-.2)node[left] {$J\cap\kappa(J)\cap J'\cap\kappa(J')$};          	
              	\draw[|-|]  (0,-.2) -- (.0840,-.2);
              	\draw[line width=.6mm]   (0,-.2) -- (.0840,-.2);          	
              	\draw[|-|]  (0.145069,-.2) -- (0.292893,-.2);
              	\draw[line width=.6mm]   (0.145069,-.2) -- (0.292893,-.2); 
              	\draw[|-|]  (0.52727,-.2) -- (0.648902,-.2);
              	\draw[line width=.6mm]   (0.52727,-.2) -- (0.648902,-.2);
              	\draw[|-|]  (2-0,-.2) -- (2-.0840,-.2);
              	\draw[line width=.6mm]   (2-0,-.2) -- (2-.0840,-.2);          	
              	\draw[|-|]  (2-0.145069,-.2) -- (2-0.292893,-.2);
              	\draw[line width=.6mm]   (2-0.145069,-.2) -- (2-0.292893,-.2); 
              	\draw[|-|]  (2-0.52727,-.2) -- (2-0.648902,-.2);
              	\draw[line width=.6mm]   (2-0.52727,-.2) -- (2-0.648902,-.2);    
        		\draw[|-|] (.9999,-.2) -- (1.00001,-.2);               	            	          	
              	          	
    			\draw[] (-.1,-.3)node[left] {$\sigma(\Delta^{\Gt})$};          	          
              	\draw[|-|]  (0,-.3) -- (.0840,-.3);
              	\draw[line width=.6mm]   (0,-.3) -- (.0840,-.3);
              	\draw[|-|] (0.179482,-.3) -- (0.292893,-.3);
              	\draw[line width=.6mm]  (0.179482,-.3) -- (0.292893,-.3);   
              	\draw[|-|] (0.52727,-.3) -- (0.626838,-.3);
              	\draw[line width=.6mm]  (0.52727,-.3) -- (0.626838,-.3); 
        		\draw[|-|] (.9999,-.3) -- (1.00001,-.3);
        		\draw[line width=.6mm](1,-.3) -- (1,-.3);
        		\draw[|-|] (1.37316,-.3) -- ( 1.47273,-.3);
        		\draw[line width=.6mm](1.37316,-.3) -- ( 1.47273,-.3);
        		\draw[|-|] (1.70711,-.3) -- (1.82052,-.3);
        		\draw[line width=.6mm](1.70711,-.3) -- (1.82052,-.3);
        		\draw[|-|] (1.91598,-.3) -- (2,-.3);
        		\draw[line width=.6mm](1.91598,-.3) -- (2,-.3); 
        		\draw[|-|] (.9999,-.3) -- (1.00001,-.3);  
        		
         		\draw[dotted]  (0,-.3) -- (0,.5);         		   
              	\draw[dotted]  (.0840,-.3) -- (.0840,.5);
              	\draw[dotted]  (0.179482,-.3) -- (0.179482,.5);
              	\draw[dotted]  (0.292893,-.3) -- (0.292893,.5);
              	\draw[dotted]  (0.52727,-.3) -- (0.52727,.5);
              	\draw[dotted]  (0.626838,-.3) -- (0.626838,.5);
              	\draw[dotted]  (1,-.3) -- (1,.5);
              	\draw[dotted]  (1.37316,-.3) -- (1.37316,.5);
              	\draw[dotted]  ( 1.47273,-.3) -- ( 1.47273,.5);
              	\draw[dotted]  (1.70711,-.3) -- (1.70711,.5);
              	\draw[dotted]  (1.82052,-.3) -- (1.82052,.5);
              	\draw[dotted]  (1.91598,-.3) -- (1.91598,.5);
              	\draw[dotted]  (2,-.3) -- (2,.5);
              	      	     	           	    	         	      			
              	\end{tikzpicture}} 
              \caption{Spectral gaps of \emph{polypropylene}. $J$ is the
                spectral localisation of the pair $\G - \{e_1\}$ and $\G
                -\{v_1\}$, and bipartiteness gives $J \cap \kappa(J)$
                as localisation set.  Similarly, $J'$ is the spectral
                localisation of the pair $\G - \{e_1\}$ and $\G -\{v_2\}$.
                Putting all this information together, we get the
                rather good spectral localisation $J \cap \kappa(J)
                \cap J' \cap \kappa(J')$.}
\label{fig:propylene}
\end{figure}


\begin{example}[Polyacetylene]
  \label{exa:poly}
  The previous examples show the existence of spectral gaps in
  periodic trees covering finite graphs with Betti number $1$. In this
  example we show how to treat more complex periodic graphs.  Consider
  the \emph{polyacetylene}, that consists of a chain of carbon atoms
  (white circles) with alternating single and double bonds between
  them, each with one hydrogen atoms (black vertex). We denote this
  graph as $\Gt$ and note that it is not a tree (cf.,
  Figure~\ref{subfig:PA}). If we want to compute $\sigma(\Delta^{\Gt})$
  we use Proposition~\ref{prp:floq.mag}. The graph $\Gt$ is
  covering the graph $\G$ (see Figure~\ref{subfig:PA_0}) which is
  bipartite and has Betti number $2$. In this case,
  \begin{equation*}
    \sigma(\Delta^{\Gt})
         = \bigcup_{t\in[0,2\pi]} \sigma(\Delta^{\G}_{\alpha^t})
  \end{equation*}
  where ${\alpha^t}$ is supported only in $e_1$ with
  $\alpha^t_{e_1}=t$ (see Remark~\ref{remark:z-periodic}).

  As in the previous examples, we assume that the vector potential
  $\alpha$ is supported only on one edge, say $e_1$.  Define
  $E_0=\{e_1\}$ and $V_0=\{v_1\}$, then $V_0$ is in the neighbourhood
  of $E_0$.  We proceed as in the previous example to localise the
  spectrum within $J\cap\kappa(J)$ (see
  Figure~\ref{fig:acetylene}). In fact, our method works almost
  perfectly in this case, since we detect precisely the spectrum
  \begin{equation*}
    J\cap\kappa(J)\setminus{\{1\}}=\sigma(\Delta^{\Gt}).
  \end{equation*}
  
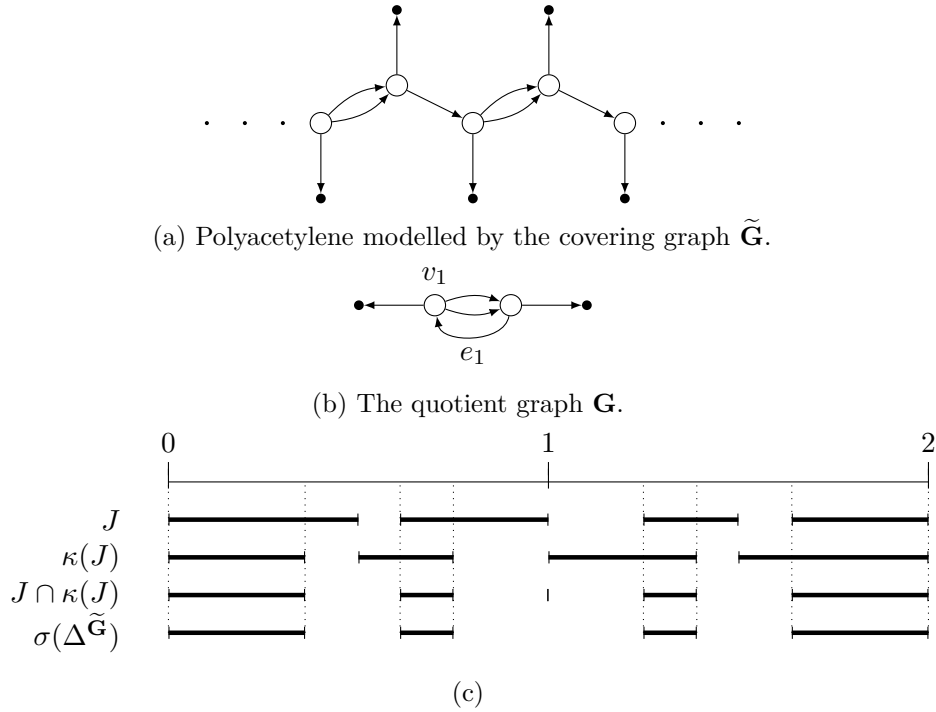
\begin{figure}[h]
  \centering 
  \subcaptionbox{Polyacetylene modelled by the covering graph $\Gt$.
    \label{subfig:PA}}%
  [.6\linewidth]{ 
    \begin{tikzpicture}[auto,
      vertex/.style={circle,draw=black!100,fill=black!100, thick,
        inner sep=0pt,minimum size=1mm},scale=2]
      \node (D) at (-1.25,0) [vertex,inner sep=.25pt,minimum
      size=.25pt,label=above:]{}; %
      \node (D) at (-1,0) [vertex,inner sep=.25pt,minimum
      size=.25pt,label=above:]{}; %
      \node (D) at (-.75,0) [vertex,inner sep=.25pt,minimum
      size=.25pt,label=above:]{}; %
      \node (O) at (-.5,0) [circle, minimum width=8pt, draw, inner
      sep=0pt] {}; %
      \node (A) at (0,.25) [circle, minimum width=8pt, draw, inner
      sep=0pt] {}; %
      \node (B) at (.5,0) [circle, minimum width=8pt, draw, inner
      sep=0pt]{}; %
      \node (C) at (1,.25) [circle, minimum width=8pt, draw, inner
      sep=0pt]{}; %
      \node (D) at (1.5,0) [circle, minimum width=8pt, draw, inner
      sep=0pt]{}; %
      \draw[-latex] (O) to[bend left=20] node[above] {} (A); %
      \draw[-latex] (O) to[bend left=-20] node[above] {} (A); %
      \draw[-latex] (A) to[] node[above] {} (B); %
      \draw[-latex] (B) to[bend left=20] node[above] {} (C); %
      \draw[-latex] (B) to[bend left=-20] node[above] {} (C); %
      \draw[-latex] (C) to[] node[above] {} (D); %
      \node (O1) at (-.5,-.5) [vertex,label=above:]{}; %
      \node (A1) at (0,.75) [vertex,label=above:]{}; %
      \node (B1) at (.5,-0.5) [vertex,label=above:]{}; %
      \node (C1) at (1,.75) [vertex,label=above:]{}; %
      \node (D1) at (1.5,-.5) [vertex,label=above:]{}; %
      \draw[-latex] (O) to[] node[above] {} (O1); %
      \draw[-latex] (A) to[] node[above] {} (A1); %
      \draw[-latex] (B) to[] node[above] {} (B1); %
      \draw[-latex] (C) to[] node[above] {} (C1); %
      \draw[-latex] (D) to[] node[above] {} (D1); %
      \node (D) at (1.75,0) [vertex,inner sep=.25pt,minimum
      size=.25pt,label=above:]{}; %
      \node (D) at (2,0) [vertex,inner sep=.25pt,minimum
      size=.25pt,label=above:]{}; %
      \node (D) at (2.25,0) [vertex,inner sep=.25pt,minimum
      size=.25pt,label=above:]{}; %
    \end{tikzpicture}}
  
  \subcaptionbox{The quotient graph $\G$.\label{subfig:PA_0}}
  [.3\linewidth]{
    \begin{tikzpicture}[auto,
      vertex/.style={circle,draw=black!100,fill=black!100, thick,
        inner sep=0pt,minimum size=1mm},scale=2]
      \node (O) at (-.5,0) [circle, minimum width=8pt, draw, inner
      sep=0pt, label=above:$v_1$] {}; %
      \node (A) at (0,0) [circle, minimum width=8pt, draw, inner
      sep=0pt] {}; %
      \draw[-latex] (O) to[bend left=20] node[above] {} (A); %
      \draw[-latex] (O) to[bend left=-20] node[above] {} (A); %
      \draw[-latex] (A) to[bend left=80] node[below] {$e_1$} (O); %
      \node (O1) at (-1,0) [vertex,label=above:]{}; %
      \node (A1) at (.5,0) [vertex,label=above:]{}; %
      \draw[-latex] (O) to[bend left=0] node[above] {} (O1); %
      \draw[-latex] (A) to[bend left=-0] node[above] {} (A1); %
    \end{tikzpicture}}\\
  \centering
  \subcaptionbox{\label{subfig:spectra}}
  [\linewidth]{	
    \begin{tikzpicture}[scale=5]
      \draw[-] (0,.5) -- (2,0.5) ; 
      \foreach \x in {0,1,2} 
      \draw[shift={(\x,0.5)},color=black] (0pt,1pt) -- (0pt,-.5pt); %
      \foreach \x in {0,1,2} 
      \draw[shift={(\x,0.5)},color=black] (0pt,0.5pt) -- (0pt,1.5pt)
      node[above] {$\x$}; %
        
      \draw[] (-.1,.4)node[left] {$J$}; %
      \draw[|-|] (0,.4) -- (0.5,.4); %
      \draw[line width=.6mm] (0,.4) -- (0.5,.4); %
      \draw[|-|] (0.609612,.4) -- (1,.4); %
      \draw[line width=.6mm] (0.609612,.4) -- (1,.4); %
      \draw[|-|] (1.25,.4) -- (1.5,.4); %
      \draw[line width=.6mm] (1.25,.4) -- (1.5,.4); %
      \draw[|-|] (1.64039,.4) -- (2,.4); %
      \draw[line width=.6mm] (1.64039,.4) -- (2,.4); %
                  	
      \draw[] (-.1,.3)node[left] {$\kappa(J)$}; %
      \draw[|-|] (0,.3) -- (0.359612,.3); %
      \draw[line width=.6mm] (0,.3) -- (0.359612,.3); %
      \draw[line width=.6mm] (0.5,.3) -- (0.75,.3); %
      \draw[|-|] (0.5,.3) -- (0.75,.3); %
      \draw[line width=.6mm] (1,.3) -- (1.39039,.3); %
      \draw[|-|] (1,.3) -- (1.39039,.3); %
      \draw[line width=.6mm] (1.5,.3) -- (2,.3); %
      \draw[|-|] (1.5,.3) -- (2,.3); %
        
      \draw[] (-.1,.2)node[left] {$J\cap\kappa(J)$}; %
      \draw[|-|] (0,.2) -- (0.359612,.2); %
      \draw[line width=.6mm] (0,.2) -- (0.359612,.2); %
      \draw[|-|] (0.75,.2) -- (0.609612,.2); %
      \draw[line width=.6mm] (0.75,.2) -- (0.609612,.2); %
      \draw[|-|] (.9999,.2) -- (1.00001,.2); %
      \draw[|-|] (2-0,.2) -- (2-0.359612,.2); %
      \draw[line width=.6mm] (2-0,.2) -- (2-0.359612,.2); %
      \draw[|-|] (2-0.75,.2) -- (2-0.609612,.2); %
      \draw[line width=.6mm] (2-0.75,.2) -- (2-0.609612,.2); %

      \draw[] (-.1,.1)node[left] {$\sigma(\Delta^{\Gt})$}; %
      \draw[|-|] (0,.1) -- (0.359612,.1); %
      \draw[line width=.6mm] (0,.1) -- (0.359612,.1); %
      \draw[|-|] (0.609612,.1) -- (0.75,.1); %
      \draw[line width=.6mm] (0.609612,.1) -- (0.75,.1); %
      \draw[|-|] (1.25,.1) -- (1.39039,.1); %
      \draw[line width=.6mm] (1.25,.1) -- (1.39039,.1); %
      \draw[|-|] (1.64039,.1) -- (2,.1); %
      \draw[line width=.6mm] (1.64039,.1) -- (2,.1); %
         	
      \draw[dotted] (0,.1) -- (0,.5); %
      \draw[dotted] (0.359612,.1) -- (0.359612,.5); %
      \draw[dotted] (0.609612,.1) -- (0.609612,.5); %
      \draw[dotted] (0.75,.1) -- (0.75,.5); %
      \draw[dotted] (1.25,.1) -- (1.25,.5); %
      \draw[dotted] (1.64039,.1) -- (1.64039,.5); %
      \draw[dotted] (2,.1) -- (2,.5); %
      \draw[dotted] (1.39039,.1) -- (1.39039,.5); %
\end{tikzpicture}}                  	      	     	           	
\caption{Spectral gaps of \emph{polyacetylene}. This is an example
  where the quotient graph has Betti number larger than $1$.  Here,
  $J$ is the spectral localisation of the pair $\G - \{e_1\}$ and $\G
  -\{v_1\}$, and bipartiteness gives again $J \cap \kappa(J)$ as
  localisation set.  The spectral localisation $J$ gives almost
  exactly the actual spectrum of $\Gt$, except for the spectral value
  $1$.}
\label{fig:acetylene}
\end{figure}

The graph $\Gt$ of Example~\ref{exa:poly} modelling polyacetylene
(see Figure~\ref{fig:acetylene}) corresponds to an intermediate
covering with respect to the maximal Abelian covering which is
graphane (see Figure~\ref{fig:Graphane}).  However, we cannot use
Theorem~\ref{theo:FSP}, but we can apply the bracketing technique to
detect the spectral magnetic gaps in $\G$ and hence spectral gaps in
graphane. Just take $E_0=\{e_1, e_2\}$ and $V_0=\{v_1\}$ being in the
neighbourhood of $E_0$, we define the bracketing intervals $J$ and
$\kappa(J)$ as before (see Figure~\ref{fig:Graphane}). Again, our
method works almost perfectly since
$J\cap\kappa(J)\setminus{\{1\}}=\sigma(\Delta^{\Gt})$.

\begin{figure}[h]
  \centering \subcaptionbox{Graphane modelled by the covering graph
    $\Gt$.\label{subfig:Graphane}}%
  [.6\linewidth]{{
      \includegraphics[width=0.3\linewidth]{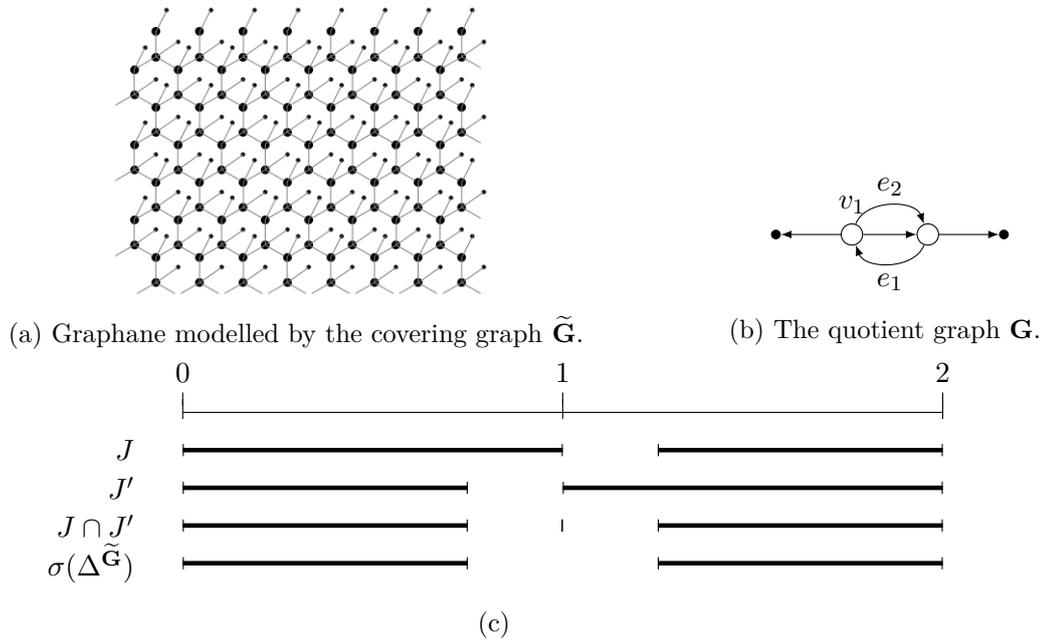}}}
  \subcaptionbox{The quotient graph $\G$.\label{subfig:Graphane_0}}
  [.3\linewidth]{ \begin{tikzpicture}[auto,
      vertex/.style={circle,draw=black!100,fill=black!100, thick,
        inner sep=0pt,minimum size=1mm},scale=2] %
      \node (O) at (-.5,0) [circle, minimum width=8pt, draw, inner
      sep=0pt, label=above:$v_1$] {}; %
      \node (A) at (0,0) [circle, minimum width=8pt, draw, inner
      sep=0pt, label=above:] {}; %
      \draw[-latex] (O) to[bend left=70] node[above] {$e_2$} (A); %
      \draw[-latex] (O) to[bend left=0] node[above] {} (A); %
      \draw[-latex] (A) to[bend left=70] node[below] {$e_1$} (O); %
      \node (O1) at (-1,0) [vertex,label=above:]{}; %
      \node (A1) at (.5,0) [vertex,label=above:]{}; %
      \draw[-latex] (O) to[bend left=0] node[above] {} (O1); %
      \draw[-latex] (A) to[bend left=-0] node[above] {} (A1); %
    \end{tikzpicture}}\\
  \centering \subcaptionbox{\label{subfig:spectra}}
  [\linewidth]{ \begin{tikzpicture}[scale=5] %
      \draw[-] (0,0.5) -- (2,0.5) ; 
      \foreach \x in {0,1,2} 
      \draw[shift={(\x,0.5)},color=black] (0pt,1pt) -- (0pt,-.5pt); %
      \foreach \x in {0,1,2} 
      \draw[shift={(\x,0.5)},color=black] (0pt,0.5pt) -- (0pt,1.5pt)
      node[above] {$\x$}; %
           	          	          
      \draw[] (-.1,.4)node[left] {$J$}; %
      \draw[|-|] (0,.4) -- (1,.4); %
      \draw[line width=.6mm] (0,.4) -- (1,.4); %
      \draw[|-|] (1.25,.4) -- (2,.4); %
      \draw[line width=.6mm] (1.25,.4) -- (2,.4); %
        	
      \draw[] (-.1,.3)node[left] {$J'$}; %
      \draw[|-|] (0,.3) -- (0.75,.3); %
      \draw[line width=.6mm] (0,.3) -- (0.75,.3); %
      \draw[line width=.6mm] (1,.3) -- (2,.3); %
      \draw[|-|] (1,.3) -- (2,.3); %
           	          	            
      \draw[] (-.1,.2)node[left] {$J\cap J'$}; %
      \draw[|-|] (0,.2) -- (0.75,.2); %
      \draw[line width=.6mm] (0,.2) -- (0.75,.2); %
      \draw[|-|] (1.25,.2) -- (2,.2); %
      \draw[line width=.6mm] (1.25,.2) -- (2,.2); %
      \draw[|-|] (.9999,.2) -- (1.00001,.2); %
           	          	             	          	    	
      \draw[] (-.1,.1)node[left] {$\sigma(\Delta^{\Gt})$}; %
      \draw[|-|] (0,.1) -- (0.75,.1); %
      \draw[line width=.6mm] (0,.1) -- (0.75,.1); %
      \draw[|-|] (1.25,.1) -- (2,.1); %
      \draw[line width=.6mm] (1.25,.1) -- (2,.1); %
    \end{tikzpicture}}
  \caption{Spectral gaps of \emph{graphane}.  This is again an example where
    the quotient graph has Betti number larger than $1$.  Here, $J$ is
    the spectral localisation of the pair $\G - \{e_1,e_2\}$ and $\G
    -\{v_1\}$ giving again a very good result for the actuall spectrum
    of $\Gt$ (except for the spectral value $1$).\label{fig:Graphane}}
\end{figure}
\end{example}

\providecommand{\bysame}{\leavevmode\hbox to3em{\hrulefill}\thinspace}


\end{document}